\documentclass[a4paper,reqno,times, 11pt]{amsart}
\usepackage{hyperref}
\usepackage[active]{srcltx}
\usepackage[curve]{xypic}
\usepackage{graphicx}
\usepackage{geometry}
\usepackage{longtable}
\usepackage{rotating}
\usepackage{multirow}
\usepackage[active]{srcltx}
\usepackage{epsfig,amsmath}
\usepackage[english]{babel}
\usepackage{amsmath,amsfonts,amssymb,amscd,color,epsfig,amsthm}
\usepackage{titletoc}
\usepackage{mathrsfs}
\usepackage{caption}

\usepackage{tikz}
\usetikzlibrary{arrows}

\newtheorem{theorem}{Theorem}[section]
\newtheorem{lemma}[theorem]{Lemma}
\newtheorem{proposition}[theorem]{Proposition}
\newtheorem{corollary}[theorem]{Corollary}

\newtheorem{definition}[theorem]{Definition}

\theoremstyle{remark}

\theoremstyle{plain}

\numberwithin{equation}{section}

\newcommand{\wt}{\widetilde}

\def\AAA{{\cal A}}

\def\EEE{{\cal E}}

\def\HHH{{\cal H}}

\def\OOO{{\cal O}}

\def\MMMM{{\cal M}}

\def\PPP{{\cal P}}

\def\MMMM{\mathscr M}
\def\HHHH{\mathscr H}
\def\CCCC{\mathscr C}
\def\PPPP{\mathscr P}

\def\g{\gamma}
\def\G{\Gamma}

\def\De{\Delta}

\def\t{\tilde}

\def\R{\mbox{$\mathbb R$}}

\def\C{\mbox{$\mathbb C$}}

\def\D{\mathbb D}
\def\Q{\mathbb Q}

\def\Z{\mbox{$\mathbb Z$}}

\def\lv{ \left(\begin{matrix} }
	\def\rv{\end{matrix}\right)}

\def\cal{\mathcal}

\def\dw{{\dw}}

\def\ov{\overline}

\title{Polynomials in  molecules}

\author{Yan Gao}
\address{Yan Gao, School of Mathematical Sciences, Shenzhen University, Shenzhen 518061, China}
\email{gyan@szu.edu.cn}

\author{Jinsong Zeng}
\address{Jinsong Zeng, School of Mathematical Sciences, Shenzhen University, Shenzhen 518061, China}
\email{jinsongzeng@163.com}

\begin{document}

	\begin{abstract}
		This paper characterizes polynomials within molecules. We show that a geometrically finite polynomial of degree
		$d\geq2$ lies in a molecule if and only if all its critical points belong to maximal Fatou chains, and show that distinct molecules are mutually disjoint. We also establish a necessary and sufficient condition for subhyperbolic polynomials to be on the closures of bounded hyperbolic components.
	\end{abstract}
	
	\thanks{
	}
	
	\maketitle
	
	
	\section{Introduction}

	Let $f: \mathbb{C} \to \mathbb{C}$ be a polynomial of degree $d \geq 2$. Its \emph{basin of infinity}
	\[
	U_f(\infty) = \{z \in \mathbb{C} : f^n(z) \to \infty \text{ as } n \to \infty\}
	\]
	is an unbounded domain satisfying $f^{-1}(U_f(\infty)) = U_f(\infty)$. The complement $K_f$ is the \emph{filled Julia set}, with boundary $J_f = \partial K_f$ defining the \emph{Julia set} and $ F_f=\mathbb{C} \setminus J_f$ the \emph{Fatou set}. The set of finite critical points of $f$ is denoted by $C_f$, and the set of \emph{postcritical points} is
	\[
	P_f = \overline{\bigcup_{n \geq 1} f^n(C_f)}.
	\]
	

	Let $\PPPP_d\simeq \C^{d-1}$ denote the space of monic and centered polynomials of degree $d\geq 2$.  The \emph{connectedness locus}
	\[\CCCC_d=\{f\in\PPPP_d: \text{the Julia set of $f$ is connected}\}\]
	is a canonical  subset of $\PPPP_d$. A central problem in complex dynamics is to study the bifurcation of connected Julia sets, which amounts to
	characterizing the structure of $\CCCC_d$.
	
	When $d=2$, the connectedness locus $\CCCC_2$ is the famous Mandelbrot set, which has been deeply studied, though the MLC (Mandelbrot locally connected) Conjecture remains open. For degrees $d\geq3$, the structure of $\CCCC_d$ becomes significantly more complex, yet it is known that $\CCCC_d$ is connected and cellular \cite{BH,La}.
	
	A class of building blocks for $\CCCC_d$ comprises \emph{bounded hyperbolic components}, i.e., the connected components of hyperbolic polynomials
	in $\CCCC_d$.
	Milnor \cite{M2} showed that each bounded hyperbolic component is a topological disk of real dimension $2d-2$. Characterizing bifurcations along these components, while simpler than the full problem, remains highly non-trivial. The subsets of $\CCCC_d$ that capture such bifurcation phenomena are referred as ``molecules''.
	
	\begin{definition}
		The set $\MMMM=\MMMM(\HHHH)$ is called a  \emph{(polynomial) molecule} if there exists  a bounded hyperbolic component $\HHHH$ such that
		\[\MMMM=\overline{\bigcup_{\HHHH'\rightsquigarrow \HHHH} \HHHH'},\]
		where the notation $\HHHH'\rightsquigarrow\HHHH$ indicates that  there exists a  sequence of  bounded hyperbolic components $\HHHH_1=\HHHH', \ldots,\HHHH_m=\HHHH $ such that $\overline{\HHHH_i}\cap \overline{\HHHH_{i+1}}\not=\emptyset$ for $i=1,\ldots,m-1$.
	\end{definition}
	
	The paper aims to describe the polynomials in molecules.  We focus particularly on the geometrically finite maps
	due to their key role in addressing the bifurcation problem.
	To state the main result, we first introduce some  notations.
	
	For any $f\in\CCCC_d$, there exists a unique conformal map $\chi_f: U_f(\infty)\to \C\setminus\overline{\D}$, called the \emph{B\"{o}ttcher coordinate}, satisfying
	$$\chi_f(z)/z\to 1\textup{ as }z\to\infty\textup{ and }\chi_f(z)^d=\chi_f\circ f(z).$$
	For any $\theta\in \R/\Z$, the set
	\[R_f(\theta)=\chi^{-1}_f(\{re^{2\pi \textbf{i}\theta}:r>1\})\]
	is called the \emph{external ray} of angle $\theta$. If the accumulation set of $R_f(\theta)$ on $J_f$ is a single point, we say that $R_f(\theta)$ \emph{lands} at this point, which is denoted by $\pi_f(\theta)$.
	By \cite[Theorem 18.10]{Mil},  the ray $R_f(\theta)$ lands when $\theta$ is rational.

	The \emph{rational lamination} $\lambda_\mathbb{Q}(f)\subseteq (\mathbb{Q}/\mathbb{Z})^2$ of $f$ consists of the pairs $(\theta_1,\theta_2)\in (\mathbb{Q}/\mathbb{Z})^2$ for which the external rays $R_f(\theta_1)$ and $R_f(\theta_2)$ land at the same point.
	
	Let $\HHHH$ be a bounded hyperbolic component. By an unpublished theorem of McMullen, there exists a unique postcritically finite polynomial $f_{0}$ contained in $\HHHH$.
	
	Suppose that $g$ is a polynomial in $\CCCC_d$ with $\lambda_\Q(f_0)\subseteq \lambda_\Q(g)$. Fix a bounded Fatou component $U$ of $f_0$. We define a connected and compact subset of $K_g$ associated to $U$ by
	\begin{equation}\label{eq:11}
		K_{g,U}=\bigcap_{(\theta, \theta')\in \lambda_{\Q}(f_0)}\overline{S_g(\theta,\theta')}\cap K_g,
	\end{equation}
	where $S_g(\theta, \theta')$ is the component of $\mathbb{C}\setminus\overline{R_g(\theta)\cup R_g(\theta')}$ which contains external rays $R_g(t)$ such that $R_{f_0}(t)$ land on $\partial U\setminus\{\pi_{f_0}(\theta)\}$.
	
	A bounded hyperbolic component  $\HHHH$ is called \emph{primitive} if  the bounded Fatou domains of $f$ have disjoint closures for $f\in \HHHH$. A point \(z\) of a connected set \(E\subseteq\C\) is called a \emph{cut point} of $E$ if \(E \setminus \{z\}\) has at least two components, a \emph{branched point} if it has at least three components, and an \emph{endpoint} if it remains connected.\vspace{8pt}
	
	\noindent{\bf Theorem A.}
	\emph{Every molecule $\MMMM$ contains a unique primitive hyperbolic component $\HHHH$ and $\lambda_\Q(f_{0})\subseteq \lambda_\Q(g)$ for every $g\in \MMMM$, where $f_0$ is the unique postcritically finite map in $\HHHH$. Furthermore, a geometrically finite polynomial $g$ lies in $\MMMM$ if and only if
		\begin{enumerate}
			\item $\lambda_\Q(f_0)\subseteq \lambda_\Q(g)$, and\vspace{1pt}
			\item  for any  Fatou domain $U$ of $f_0$,  the set $K_{g,U}$ is  a maximal Fatou chain of $g$, or equivalently,
			$K_{g,U}$ has at most countably  many  cut points.
		\end{enumerate}
	}
	\vspace{8pt}

	\noindent{\bf Remark.} We add the following remarks regarding Theorem A.\vspace{3pt}

	(1) The notion of (maximal) Fatou chains was originally introduced in \cite{CGZ,DHS} to study the combinatorics of Julia sets  for rational maps. It may apply to polynomials with minor adaptation, see Section \ref{sec:chain}. Theorem A implies that:\vspace{2pt}
	
	\emph{A geometrically finite polynomial lies in a molecule if and only if each of its critical points is contained in a maximal Fatou chain.}\vspace{2pt}
	
	Thus, maximal Fatou chains correspond precisely to molecules in parameter space.\vspace{3pt}
	
	(2) Condition (2) of Theorem A is equivalent to a finite condition:  $K_{g,U}$ is a maximal Fatou chain for each periodic Fatou component $U$ of $f_0$.\vspace{3pt}
	
	(3) Theorem A shows that each molecule $\MMMM$ is closed under quasiconformal deformation for geometrically finite maps. That is, if a geometrically finite $f\in\MMMM$ is quasiconformally conjugate to some $g\in\PPPP_d$, then $g$ also lies in $\MMMM$. However, this property is known to fail for the closure of a single bounded hyperbolic component, see \cite{Tan}.
	\vspace{3pt}
	
	As we will see in Section \ref{sec:chain}, the maximal Fatou chains of a polynomial are pairwise disjoint. Their counterparts in parameter space, the molecules, share this property.\vspace{8pt}
	
	\noindent {\bf Theorem B.}
	\emph{Two distinct molecules are always disjoint.
	}\vspace{8pt}
	
	The unique hyperbolic component containing $z \mapsto z^d$ is called the \emph{main hyperbolic component} $\HHHH(d)$. The molecule containing $\HHHH(d)$ is called the \emph{main molecule}, denoted by $\MMMM(d)$.
	
	Luo and Park \cite{LP}  undertook the first systematic study of the so-called extended main molecule.
	For a postcritically finite polynomial $f\in\CCCC_d$, let $\HHHH_f$ denote the connected component  of
	\begin{equation}\label{eq:111}
		\{g\in \PPPP_d:\ g\text{ is quasiconformal conjugate to }f\text{ near the Julia set}\}
	\end{equation}
	that contains $f$, referred to as a \emph{relative hyperbolic component} in \cite{LP}. Note that $\HHHH_f=\{f\}$ if $f$ has no bounded Fatou domains, and $\HHHH_f$ is a hyperbolic component if $f$ is hyperbolic.

	They defined the \emph{extended main molecule} as
	\[{\MMMM}_+(d)=\ov{\bigcup_{\HHHH_f\twoheadrightarrow\HHHH(d)} \HHHH_f},\]
	where $f$ is a postcritically finite polynomial,  and the notion $\HHHH_f\twoheadrightarrow\HHHH(d)$ means that
	there exists a finite sequence of postcritically finite polynomials $f=f_1,\ldots,f_m(z)= z^d$ such that $\overline{\HHHH_{f_i}}\cap \overline{\HHHH_{f_{i+1}}}\not=\emptyset$ for all $i=1,\ldots,m-1$.\vspace{3pt}
	
	One main result of \cite{LP} is as follows.\vspace{8pt}
	
	\noindent {\bf Theorem LP} (\cite[Theorm 1.1]{LP}).  {\it Let $f\in\PPPP_d$ be a postcritically finite polynomial.
		Then the following statements are equivalent.
		\begin{enumerate}
			\item The topological entropy of $f$ restricted to its Hubbard tree $T_f$ is zero;
			\item $f$ is contained in the extended main molecule $\MMMM_+(d)$;
			\item $f$ is obtained from $z^d$ via a finite sequence of simplicial tunings;
			\item $J_f\cap T_f$ is at most a countable set.
		\end{enumerate}
	}
	\vspace{5pt}
	
	It is clear that  $\MMMM(d)\subseteq {\MMMM}_+(d)$, and it is known that $\MMMM(2)={\MMMM}_+(2)$. A natural question then arises: does the equality $\MMMM(d)={\MMMM}_+(d)$ hold for all $d\geq3$? \vspace{2pt}
	
	Since the lamination of $\HHHH(d)$ is trivial, it follows from Theorem A  that, a geometrically finite map $f$ lies in $\MMMM(d)$ if and only if $K_f$ has countably many cut points. This condition is in fact equivalent to property (4) in Theorem LP, see Proposition \ref{pro:iterate}. We thus obtain the following.

	\begin{corollary}\label{coro:3}
		For every degree $d\geq 2$, we have $\MMMM(d)={\MMMM}_+(d)$.
	\end{corollary}
	
	This corollary yields the density of  hyperbolic maps in ${\MMMM}_+(d)$, thereby positively answering Question 1.6 raised in \cite{LP}. \vspace{2pt}
	
	Theorem A classifies the geometrically finite polynomials within a molecule. A  subsequent question is: which of these polynomials lie in the closure of a bounded hyperbolic component?

	For $d=2$, every geometrically finite map in a molecule
	lies in the closure of a bounded hyperbolic component. For higher degrees,  however, the situation is more subtle.
	As shown in \cite[Section 3.2]{LP},  there exists a postcritically finite map in $\MMMM(3)$
	outside the closure of any bounded hyperbolic component.
	In addition, the combinatorial classification of geometrically finite polynomials on $\partial\HHHH(d)$ was established in \cite{Luo}.\vspace{2pt}
	
	Our third main result completely characterizes the subhyperbolic polynomials in the closures of bounded hyperbolic components.\vspace{8pt}
	
	\noindent{\bf Theorem C.}
	\emph{Let $f\in\CCCC_d$ be a subhyperbolic polynomial. Then $f$ lies in the closure of a bounded hyperbolic component if and only if every critical point of $f$ is eventually mapped into the maximal Fatou chain generated by a single periodic bounded Fatou domain.
	}\vspace{8pt}
	
	For  definition of maximal Fatou chains generated by single Fatou domains, see Definition~\ref{def:chain}.
	\vspace{3pt}

	The proofs of the above theorems rely on three main tools.\vspace{3pt}
	
	$\bullet$ The first is the combinatorial notion of Fatou chains mentioned earlier. To our knowledge, this is the first application of Fatou chains in the study of parameter spaces.\vspace{3pt}

	$\bullet$ The second  is the hyperbolic-parabolic deformation theory established in \cite{CT2}. \vspace{3pt}
	
	$\bullet$ The third  is a new dynamical perturbation technique for geometrically finite polynomials (Theorem~\ref{thm:perturbation}), which provides a hyperbolic–subhyperbolic deformation framework, thus complementing the hyperbolic-parabolic deformation theory when studying the polynomial bifurcations.
	\vspace{3pt}

	The organization of the paper is as follows.\vspace{2pt}
	
	In Section 2 we define two kinds of Fatou chains and establish their basic properties.  Section 3 proves the dynamical perturbation theorem (Theorem  \ref{thm:perturbation}) for geometrically finite polynomials, and Section 4 introduces laminations and geodesic laminations	
	
	We show in Section 5 that the conclusion of Theorem A holds for ``extended molecules'' (Theorem \ref{thm:gf}), and then prove in Section 6 that extended molecules coincide with molecules (Theorem \ref{thm:same}). This completes the proof of Theorems A and B.
	Theorem C is proved in Section 7.

	\vspace{15pt}
	
	\noindent{\bf Acknowledgement.} The authors are grateful for insightful discussions with Guizhen Cui and Luxian Yang. The first author is supported by the  National Key R\&D Program of China (Grant no. 2021YFA1003203), the NSFC (Grant no. 12322104)
	and the NSFGD (Grant no. 2023A1515010058). The second author is supported by the NSFC (Grant
	no. 12271115).
	
	\section{Fatou chains of polynomials}\label{sec:chain}

	A \emph{continuum} is a connected and compact subset of $\mathbb{C}$ containing at least two points, and a continuum is \emph{full} if its complement is connected. A \emph{disk} always means a Jordan domain in $\C$.	
	
	Let $f$ be a polynomial in $\CCCC_d$. For a continuum $E$ satisfying $f^p(E)\subseteq E$, we define the \emph{continuum generated by} $E$ as
	\[K_E=\ov{\bigcup_{i\geq 0} E^i},\] where $E^i$ denotes the component of $f^{-ip}(E)$ containing $E$. Since $E^i\subseteq E^{i+1}$ for all $i\ge0$, the set $K_E$ remains unchanged whether we use the original period $p$ or any multiple $p'=kp$.
	
	If $K_E=E$, or equivalently, if $E$ is a connected component of $f^{-p}(E)$, then $E$ is called a \emph{stable continuum} of $f$.
	We say that a full stable continuum $E$ \emph{induces  renormalization} if there exists disks $W\Subset V$ such that $f^p:W\to V$ is proper and $E=\{z\in W: f^{ip}(z)\in W,\forall\, i\geq0\}$.
	
	The following is  a criterion when a full stable continuum induces  renormalization.
	
	\begin{lemma}[\cite{BOPT},\,Theorem B]\label{lem:renormalization}
		Let $E$ be a full stable continuum of $f$ with  $\mathrm{deg} f|_E\geq 2$. Then, $E$ induces  renormalization if and only if for every attracting or parabolic point $z\in E$ of $f$, the immediate attracting
		basin of $z$ or the union of all immediate parabolic domains of $z$ is a subset of $E$.
	\end{lemma}

	%
	
	We now recursively define the (periodic) Fatou chains of $f$ in each level:
	\begin{itemize}
		\item \textbf{Level-0:} The closures of all bounded periodic Fatou domains.
		
		\item \textbf{Level-1:} Let $E_1, \ldots, E_m$ be the components of the union of all  level-0 Fatou chains. The level-1  Fatou chains are  defined as the continua generated by $E_1, \ldots, E_m$, respectively.
		
		\item \textbf{Level-$\mathbf{(n+1)}$:} Assume that the level-$n$ periodic Fatou chains have been defined. Let $E_1, \ldots, E_m$ be the components of their union. The level-$(n+1)$  Fatou chains are defined as the continua generated by these components.
	\end{itemize}

	\begin{proposition}\label{pro:max}
		There exists an integer $N$ such that each level-$N$ periodic Fatou chain is also a level-$(N+1)$ one. Consequently, each level-$N$ periodic Fatou chain is a stable continuum of $f$.
	\end{proposition}
	\begin{proof}
		Let $k(n)$ denote the number of level-$n$ Fatou chains of $f$. Since $k(n)$ is decreasing, there exists $n_0$ such that $k(n)$ is constant for $n \geq n_0$, and distinct level-$n$ chains are pairwise disjoint.
		
		For a periodic bounded Fatou domain $U$ of period $p$, let $K_n(U)$ be the level-$n$ Fatou chain containing $U$. Then $f^p(K_n(U)) = K_n(U)$, and for $n \ge n_0$, $K_n(U)$ is the unique level-$n$ Fatou chain in $K_{n+1}(U)$. If $K_n(U)$ is not a component of $f^{-p}(K_n(U))$, then
		\[
		\deg(f^p|_{K_{n+1}(U)}) > \deg(f^p|_{K_n(U)}).
		\]
		Since $\deg(f^p|_{K_{n+1}(U)}) \leq (\deg f)^p$, the sequence $\deg(f^p|_{K_n(U)})$ must stabilize. Thus, there exists $n(U)$ such that for $n \ge n(U)$, $K_n(U)$ is a component of $f^{-p}(K_n(U))$ and $K_{n+1}(U) = K_n(U)$.
		
		Taking $N = \max n(U)$ over all periodic bounded Fatou domains $U$, every level-$n$ Fatou chain for $n \geq N$ is also a level-$N$ Fatou chain.
	\end{proof}
	These level-$N$ Fatou chains together with all components of their iterated preimages form the \emph{maximal Fatou chains} of $f$, which are pairwise disjoint by construction.
	
	\begin{proposition}\label{pro:cut-point}
		Suppose that $J_f$ is locally connected. Then
		each Fatou chain of any level contains at most countably many cut points, with only finitely many being periodic. Furthermore, if each critical point in $J_f$ is preperiodic,  all these cut points are preperiodic.
	\end{proposition}
	\begin{proof}
		The proof goes by induction on the level of Fatou chains.
		Clearly, level-$0$ Fatou chains contain no cut points. Assume that every level-$n$ Fatou chain satisfies the stated property. Let $E$ be a component of the union of all  level-$n$ Fatou chains, and denote by $B_1,\ldots,B_m$ the level-$n$ Fatou chains contained in $E$.
		
		Let $z$ be a cut point of $E$ that is not a cut point of any $B_i$. Then $z \in B_i \cap B_j$ for some $i \neq j$. Since $J_f$ is locally connected, there exist external rays $R_f(\alpha)$ and $R_f(\beta)$ landing at $z$ such that $\gamma = R_f(\alpha) \cup \{z\} \cup R_f(\beta)$ separates $B_i$ and $B_j$, which implies $\{z\} = B_i \cap B_j$. As $B_i$ and $B_j$ are periodic, $z$ is periodic. Combined with the induction hypothesis, this shows $E$ satisfies the  conclusion of the proposition.
		
		For $k \geq 1$, let $E_k$ be the component of $f^{-kp}(E)$ containing $E$, where $p$ is the period of $E$. Any cut point of $E_k$ maps under iteration to either a critical point of $f^p$ or a cut point of $E$. Therefore, every $E_k$, and hence $\bigcup_{k \geq 0} E_k$, satisfies the  conclusion of the proposition.
		
		Let $K = \overline{\bigcup_{k \geq 0} E_k}$ be the level-$(n+1)$ Fatou chain generated by $E$. We claim that no point $x \in K \setminus \bigcup_{k \geq 0} E_k$ is a cut point of $K$. Suppose otherwise: then there exist external rays $R_f(\theta_\pm)$ landing at $x$ such that both components $W_\pm$ of $\mathbb{C} \setminus (R_f(\theta_+) \cup \{x\} \cup R_f(\theta_-))$ intersect $K$. Since $x$ avoids all $E_k$ ($k \geq 1$), the set $\bigcup_{k \geq 0} E_k$ lies entirely in one component, say $W_+$. Then $K $ is disjoint from $W_-$, a contradiction. Thus $K$ satisfies the required property.  \end{proof}
	
	
	Similarly, we  can  define the   Fatou chains generated by  single Fatou domains.
	
	\begin{definition}\label{def:chain}
		Let $U$ be a bounded periodic Fatou domain of $f$. The level-$0$ Fatou chain generated by $U$ is defined as $\overline{U}$. Inductively, if $E$ denotes the level-$n$ Fatou chain generated by $U$, then the level-$(n+1)$ Fatou chain generated by $U$ is defined as the continuum generated by $E$.
	\end{definition}
	
	Analogous to Proposition \ref{pro:max}, this process stabilizes, yielding a \emph{maximal Fatou chain generated by $U$}, which is still a stable continuum of $f$.
	
	\begin{proposition}\label{pro:chain1}
		Suppose that $J_f$ is locally connected. Then the maximal Fatou chain $E_U$ generated by a fixed Fatou domain $U$ has no periodic cut points. In particular, $E_U=\overline{U}$ if and only if $\partial U$ contains no critical points.
	\end{proposition}
	\begin{proof}
		Note that $\overline{U}$, the level-$0$ Fatou chain generated by $U$, has no cut points. By induction, assume the level-$n$ Fatou chain $E$ generated by $U$ has no periodic cut points.
		
		For each $k \geq 1$, let $E_k$ be the component of $f^{-k}(E)$ containing $E$.
		We claim that $E_k$ has no periodic cut points.
		Suppose otherwise: let $z$ be a periodic cut point of $E_k$, and let $R_f(\theta_\pm)$ be external rays landing at $z$ such that $\gamma = R_f(\theta_+) \cup \{z\} \cup R_f(\theta_-)$ separates $E_k$. Since $f^k$ is injective near $z$ and $f^k(E_k) = E$, it follows that $z$ would disconnect $E$, a contradiction.
		
		Thus, the level-$(n+1)$ Fatou chain $K = \overline{\bigcup_{k \geq 0} E_k}$ generated by $U$ has no periodic cut points.
	\end{proof}
	
	Using Lemma \ref{lem:renormalization}, we immediately obtain the following.
	
	\begin{proposition}\label{pro:renormalization}
		If $f$ has no irrational indifferent periodic points, then any periodic maximal Fatou chain induces  renormalization. If $f$ has no indifferent periodic points, then  any maximal Fatou chain generated by a single periodic Fatou domain induces renormalization.
	\end{proposition}

	Next, we present several equivalent characterizations of   stable continua  in  maximal Fatou chains for geometrically finite polynomials.
	Factually, any combinatorial issue can be  first addressed in the postcritically finite case and then extended to geometrically finite case via the following result, established independently by Kawahira \cite{Ka} and Cui-Tan \cite{CT2} using different methods.
	
	\begin{theorem}\label{thm:CT}
		Let $f$ be a geometrically finite polynomial in $\CCCC_d$. Then, there exists a unique postcritically finite polynomial $g\in\PPPP_d$, and a homeomorphism $\phi:\C\to\C$ with $\phi(z)/z\to 1$ as $z\to\infty$, such that $\phi(J_f)=J_g\textup{ and }\phi\circ f=g\circ \phi\textup{ on }J_f.$
	\end{theorem}
	
	For a  postcritically finite polynomial $g$, its {\it Hubbard tree} $T_g$  is the minimal $g$-invariant and finite regulated tree in $K_g$ containing $P_f$; see  \cite[Proposition 2.7]{DH1}. Here, a finite tree $T\subseteq K_g$ is called \emph{regulated} if for any edge of $T$, its intersection with the closure of any Fatou domain consists of one or two internal rays together with the landing points. It is well-known that
\begin{proposition}\label{pro:iterate}
Every cut point of $J_g$ is eventually iterated into $T_g$.
\end{proposition}
	The Hubbard tree $T_f$ for a geometrically finite polynomial $f\in\CCCC_d$ is then defined as a finite tree in $K_f$ such that $P_f\subseteq T_f$ and $T_f\cap J_f=\phi^{-1}(T_g\cap J_g)$. Here $\phi$ and $g$ are the maps obtained in Theorem \ref{thm:CT}.
	
	\begin{proposition}\label{pro:cluster}
		Let $K$ be a full stable continuum of  a geometrically finite polynomial $f\in\CCCC_d$ with $\partial K\subseteq J_f$. Then the following statements are equivalent:
		\begin{enumerate}
			\item $K$ is contained in a maximal Fatou chain of $f$;
			\item each cut point of $K$ is preperiodic;
			\item $K$ contains at most countably many cut points.
						\item $K$ contains finitely many periodic cut points;
			
		\end{enumerate}
	\end{proposition}
	
	The proof of this proposition is based on the following lemma.

	\begin{lemma}\label{lem:non-cluster}
		Let $f$ be a postcritically finite polynomial. Then, given any regulated arc $I_0\subseteq K_f$ not lying in any maximal Fatou chain, there exist arcs $I'\subseteq I_0$, $ I=[x,y]\subseteq K_f$, two disjoint sub-arcs $I_1, I_2\subseteq I$ and positive integers $l, m_1, m_2$, such that
		\begin{enumerate}
			\item $x,y$ are pre-repelling and their orbits avoid the critical points;
			\item $f^l:I'\to I$ is a homeomorphism;
			\item $I$ is disjoint from $P_f$ and branched points of the Hubbard tree $T_f$;
			\item the maps $f^{m_k}:I_k\to I$ are homeomorphisms for $k=1, 2$.
		\end{enumerate}
		As a consequence, for any $n>0$, there exists a preperiodic point $z_n\in I'$ whose period exceeds $n$.
	\end{lemma}
	A result similar to this lemma is proved in \cite[Lemma 6.4]{LP} by Cantor-Bendixson decomposition of graphs. Here we present a direct proof based on the approach in \cite[Lemma 9.14]{GT}.
	\begin{proof}
		Let $X_1$ be the union of all maximal Fatou chains of $f$ containing critical or postcritical points, and let $X_2$ be the set of branch points of $T_f$. Define $X = X_1 \cup X_2 \cup P_f$. Then $f(X) \subseteq X$.
		
		Let $Y$ be the set of points in $J_f$ whose orbits avoid $X$. Since $I_0$ is not contained in any maximal Fatou chain, we conclude that $K_f$ is not a maximal Fatou chain. It follows that $Y$ is non-empty and contains no isolated points.
		
		Since $f$ is postcritically finite, there exists $\lambda > 1$ such that for any arc $\gamma \subseteq K_f$ with $f^n: \gamma \to K_f \setminus X$ injective, we have
		\begin{equation}\label{eq:expanding}
			\mathrm{diam}_\omega(f^n(\gamma)) \geq \lambda^n \, \mathrm{diam}_\omega(\gamma),
		\end{equation}
		where $\mathrm{diam}_\omega(\cdot)$ denotes the homotopic diameter with respect to the canonical orbifold metric of $f$ (see \cite[Appendix A]{CGZ}).
		
		Let $\Gamma$ be a maximal collection of pairwise disjoint open arcs $\beta \subseteq T_f \setminus X$
such that one endpoint of $\beta$ belongs to $X$ and the other to $Y$, and that $\mathrm{diam}_\omega(\beta) < \kappa$ for a sufficiently small constant $\kappa$.
		Then $\Gamma$ is finite, and each $Y \cap \beta$ is non-empty and has no isolated points.
		
		Let $\gamma \subseteq K_f$ be any regulated arc not contained in a maximal Fatou chain of $f$. Then $\gamma \cap Y \neq \emptyset$.\vspace{5pt}

		{\bf Claim}.   \emph{For any $y \in Y \cap \gamma$ and any $\epsilon > 0$, there exist an open arc $\alpha \Subset \gamma$ in the $\epsilon$-neighborhood of $y$ and $\beta \in \Gamma$ such that $f^m: \alpha \to \beta$ is a homeomorphism.}

		\begin{proof}[Proof of the claim.]
			Choose $n_\epsilon$ such that $\kappa / \lambda^{n_\epsilon} < \epsilon/2$.
			By expansivity \eqref{eq:expanding} and Proposition \ref{pro:iterate}, there exist an open arc $\alpha_0 \subseteq \gamma$ containing $y$ and $m \geq n_\epsilon$, such that $f^m(\alpha_0)$ lies in $T_f$ and intersects $X$. Since $f^m(y) \notin X$, some $\beta \in \Gamma$ intersects $f^m(\alpha_0)$. Let $\alpha$ be the component of $f^{-m}(\beta)$ intersecting $\alpha_0$. Then $\mathrm{diam}_\omega(\alpha) < \epsilon/2$ by \eqref{eq:expanding}. As $\beta$ is disjoint from $P_f$, $f^m$ is injective near $\alpha$. Since $\beta$ avoids branch points of $T_f$, we have $\alpha \subseteq \gamma$.
		\end{proof}
		
		Let $N = \#\Gamma$. Fix any $y \in \gamma \cap Y$ and $\epsilon > 0$. Successively applying the claim, we obtain nested sub-arcs $\alpha_1 \Supset \alpha_2 \Supset \cdots \Supset \alpha_{N+1}$ of $\gamma$ in the $\epsilon$-neighborhood of $y$ and elements $\beta_1, \ldots, \beta_{N+1} \in \Gamma$ such that $f^{n_k}: \alpha_k \to \beta_k$ is a homeomorphism with $n_k < n_{k+1}$.
		
		By the pigeonhole principle, there exist $i < j$ such that $\beta_i = \beta_j$. Set $\beta = \beta_i$, $m = n_j - n_i$, and $\alpha = f^{n_i}(\alpha_j)$. Then $\alpha \Subset \beta$ and $f^m: \alpha \to \beta$ is a homeomorphism. Thus $f^m$ has a repelling fixed point in $\alpha$ whose orbit avoids $X$. This implies $\alpha_j$ contains a pre-repelling point in $Y$. Varying $y$ and $\epsilon$, we conclude that
		\begin{enumerate}
			\item pre-repelling points are dense in $Y \cap \gamma$; and
			\item the periods of these pre-repelling points tend to infinity.
		\end{enumerate}
		
		By  the discussion above, there exist arcs $\alpha \subseteq \gamma$ and $\beta \in \Gamma$ such that $f^l: \alpha \to \beta$ is a homeomorphism, and $\beta$ contains periodic points $z_1, z_2 \in Y \cap \beta$ with large periods $m_1, m_2$. Since $\gamma$ is arbitrary, Statement (1) also applies to $\beta$. Thus we obtain pre-repelling points $x, y \in Y \cap \beta$ with $z_1, z_2 \in I = [x,y]$. By expansivity \eqref{eq:expanding}, for $i = 1, 2$, the component $I_i$ of $f^{-m_i}(I)$ containing $z_i$ satisfies $I_i \Subset I$ and $\overline{I_1} \cap \overline{I_2} = \emptyset$.
		
		Finally, set $\gamma = I_0$ and $I' = (f^l_\alpha)^{-1}(I)$. This completes the proof of the lemma.
	\end{proof}

	\begin{proof}[Proof of Proposition \ref{pro:cluster}]
	We claim that for any $x,y\in K$, the regulated arc $[x,y]\subseteq K_f$ is contained in $K$.
Suppose this fails. Then $\mathbb{C}\setminus\big([x,y]\cup K\big)$ contains a bounded domain $D$, which must lie in $F_f$.
Since $\partial K\subset J_f$, the part of $\partial D\setminus [x,y]$ is contained in the boundary of some Fatou domain.
Because $K$ is $f$-invariant, there exists a periodic Fatou domain $U$ with $\partial U\subseteq K$.
As $K$ is full, $\overline{U}\subseteq K$. The stability of $K$ then implies $D\subseteq K$, a contradiction.
Thus the claim holds.

	This claim shows that every cut point of $K$ is also a cut point of $K_f$. 	Then combining Proposition \ref{pro:cut-point},  we have  $(1)\Rightarrow (2)\Rightarrow (3)$ and $(1)\Rightarrow (4)$.
	
	On the other hand, if $K$ is not contained in a maximal Fatou chain, the claim above implies that $K$ contains a regulated arc of $K_f$ that does not lie in any maximal Fatou chain. It then follows from
	 Lemma \ref{lem:non-cluster} that $K$ contains uncountably many cut points and infinitely many periodic cut points. Therefore, Statements (3) and (4) fail.
	\end{proof}

	\section{Dynamical Perturbation of polynomials}\label{sec:pinching}

	In this section, we develop a surgical method to perturb a critical point from the Julia set into the Fatou set while preserving the underlying dynamics (Theorem \ref{thm:perturbation}).
	This construction will play a key role in the proof of Theorems A, B and C. 
	
	\subsection{Parabolic points and convergence of external rays}
	Let $D\ni 0$ be a disk in $\C$ and $f:D\to \C$ be a univalent holomorphic function with $f(0)=0$. Suppose that $f'(0)=1$. Then 
	\[f(z)=z+z^{\nu+1}+o(z^{\nu+1}),\ \text{as }z\to 0.\]
The integer $\nu\geq1$ is called the \emph{multiplicity} of $0$ as a parabolic fixed point. The dynamical structure of $f$ near $0$ is described by the  Leau-Fatou Flower Theorem, see \cite[Section 10]{Mil}.

\begin{theorem}[Leau-Fatou Flower Theorem]\label{thm:flower}
There exist $\nu$ pairwise disjoint disks $\PPP_1^+,\ldots,\PPP_\nu^+\subseteq D$, called \emph{repelling petals}, and $\nu$ pairwise disjoint disks $\PPP_1^-,\ldots,\PPP_\nu^-\subseteq D$, called \emph{attracting petals}, for $f$ at $0$, such that:
\begin{enumerate}
\item The closures of any two distinct repelling  (resp.\,attracting) petals intersect only at $0$, and these $2\nu$ petals together with $0$ forms an open neighborhood of $0$;
\item For each $k\in\{1,\dots,\nu\}$, $f(\PPP_k^-)\subseteq \PPP_k^-$ and $f^n\to0$ uniformly on $\PPP_k^-$ as $n\to\infty$. Moreover, every orbit of $f$ converging to $0$ eventually enters some $\PPP_k^-$;
\item For each $k\in\{1,\dots,\nu\}$, $f^{-1}(\PPP_k^+)\subseteq \PPP_k^+$ and $(f^{-1})^n\to0$ uniformly on $\PPP_k^+$ as $n\to\infty$. Moreover, every orbit of $f^{-1}$ converging to $0$ eventually enters some $\PPP_k^+$.
\end{enumerate}
Moreover, let $\{f_n:D\to\C\}$ be a sequence of holomorphic functions converging uniformly to $f$, such that each $f_n$ has a parabolic fixed point at $0$ with multiplier $1$ and multiplicity $\nu$. Then, for each sufficiently large $n$, there exist attracting petals $\PPP_1^-(f_n),\dots,\PPP_\nu^-(f_n)$ and repelling petals $\PPP_1^+(f_n),\dots,\PPP_\nu^+(f_n)$ for $f_n$ at $0$ such that the closure of $\PPP_k^\pm(f_n)$ converges to that of $\PPP_k^\pm$ as $n\to\infty$ in the Hausdorff metric, for each $k\in\{1,\dots,\nu\}$.
\end{theorem}

If $f'(0)=e^{2\pi iq/p}\not=1$, where $0<q<p$ are coprime integers, then $(f^p)'(0)=0$ and we define the \emph{multiplicity of $f$ at $0$} to be the multiplicity of $f^p$ at $0$.

\begin{lemma}\label{lem:stable}
Let $\{f_n\}\subseteq\CCCC_d$ be a sequence converging to $g$, and let $z$ be a preperiodic point of $g$.
\begin{enumerate}
\item If $z$ is pre-repelling, then for every $\theta$ with $\pi_g(\theta)=z$, we have  $\overline{R_{f_n}(\theta)}\to\overline{R_g(\theta)}$ as $n\to\infty$ and  $\pi_{f_n}(\theta)$ is pre-repelling for all large $n$. If, in addition, $z$ is not pre-critical, then  $\pi_g(\theta)=\pi_g(\eta)=z$ implies $\pi_{f_n}(\theta)=\pi_{f_n}(\eta)=z_n$ for all sufficiently large $n$.\vspace{2pt}

\item If $z$ is a parabolic periodic point, and if each $f_n$ admits a parabolic point $z_n$ with the same period, multiplier and multiplicity as $z$, such that $z_n\to z$, then for every $\theta$ with $\pi_g(\theta)=z$, we have $\pi_{f_n}(\theta)=z_n$ for all sufficiently large $n$, and $\overline{R_{f_n}(\theta)}\to\overline{R_g(\theta)}$ as $n\to\infty$.
\end{enumerate}\end{lemma}
	\begin{proof}
Statement (1) is a well‑known result, see for instance \cite[Proposition 8.1]{DH1}. Statement (2) is proved by a similar argument.

After iteration we may assume that \(f(z)=z\), \(f_n(z_n)=z_n\) and \(f'(z)=f_n'(z_n)=1\).  
Suppose the external ray  
\[
R_f(\theta)=\{\chi_f^{-1}(re^{2\pi i\theta})\mid r>1\}
\]  
lands at \(z\), where \(\chi_f\) is the Böttcher coordinate of \(f\) and \(r\) is the \emph{potential} of \(\chi_f^{-1}(re^{2\pi i\theta})\).  
By Theorem~\ref{thm:flower} there exist a repelling petal \(\PPP\) of \(f\) at \(z\) and repelling petals \(\PPP_n\) of \(f_n\) at \(z_n\), such that the points of \(R_f(\theta)\) with potential \(\le r_0\) lie in \(\PPP\) and \(\overline{\PPP_n}\to\overline{\PPP}\) as \(n\to\infty\).

For each \(n\), let \(\gamma_n\) be the sub‑arc of \(R_{f_n}(\theta)\) with \(\gamma_n(0)\) having potential \(r_0/d\) and \(\gamma_n(1)\) having potential \(r_0\). Then \(f_n(\gamma_n(0))=\gamma_n(1)\). Since \(\chi_{f_n}\to\chi_f\) locally uniformly, we have \(\gamma_n\subseteq\PPP_n\) for all sufficiently large \(n\).  
Fix such an \(n\) and set \(\gamma_{n,0}=\gamma_n\). Inductively, for \(k\ge1\) let \(\gamma_{n,k}\) be the lift of \(\gamma_{n,k-1}\) by \(f\) based at \(\gamma_{n,k-1}(0)\). The concatenation \(\Gamma_n\) of the arcs \(\gamma_{n,k},\;k\ge0,\) consists precisely of those points of \(R_{f_n}(\theta)\) with potential \(\le r_0\). Theorem~\ref{thm:flower}\,(3) gives \(\Gamma_n\subseteq\PPP_n\) and converging to $z_n$, which completes the proof of (2).
\end{proof}

	\subsection{Dynamical Perturbation Theorem}
	By a \emph{last critical point} of a polynomial $f$, we mean a critical point $c$ in $J_f$ such that $f^k(c)\notin C_f$ for all $k\geq 1$. The collection of maximal Fatou chains of $f$ that contain  critical or postcritical points of $f$ is denoted by $\EEE_f$.
	
	\begin{theorem}[Dynamical perturbation]\label{thm:perturbation}
		Let $f\in\CCCC_d$ be  geometrically finite with non-empty bounded Fatou domains. Let $c^*$ be a last critical point of $f$. Then there exists a sequence of geometrically finite maps $\{f_n\}\subseteq \CCCC_d$ uniformly converging to $f$ such that:
		\begin{enumerate}
			\item \textup{(Critical point correspondence)} For every $c\in C_f$,  $f_n$ has a unique critical point $c_n$ with  ${\rm deg}(f,c)={\rm deg}(f_n, c_n)$ and $c_n\to c$ as $n\to\infty$, such that $c_n^*\in F_f$, and  $c_n\not=c_n^*$ lies in a parabolic domain or the boundary of a bounded Fatou domain precisely when $c$ does.\vspace{3pt}
			
			\item \textup{(Critical relation consistency)} For distinct $c,c'\in C_f$, if $f^k(c)=c'$, then $f_n^k(c_n)=c_n'$;
			if $c$ is not iterated to $c^*$, then $f^{m+l}(c)=f^{m}(c)$ implies $f_n^{m+l}(c_{n})=f_n^m(c_{n})$. \vspace{3pt}
			\item\textup{(Parabolic point correspondence)} For each parabolic periodic point $z$ of $f$, there is a unique parabolic point $z_n$ of $f_n$ converging to $z$ as $n\to\infty$, and sharing the same multiplier, period, and multiplicity as $z$ does for $f$.			
			
			\item \textup{(Fatou chain correspondence)} If $c^*$ lies in a maximal Fatou chain of $f$, then there exists a bijection $\iota_n:\EEE_f\to\EEE_{f_n}$ such that\vspace{2pt}
			\begin{itemize}
				\item $c\in K$ if and only if $c_n\in\iota_n(K)$ for any $c\in C_f$ and $K\in\EEE_f$;
				\item $\iota_n\circ f(K)=f_n\circ \iota_n(K)$ for any $K\in\EEE_f$.
			\end{itemize}
			\vspace{3pt}
			
			\item \textup{(Quasiconformal deformation and semiconjugacy)} If $c^*$ lies in the boundary of  a bounded Fatou domain, then there exists a polynomial $\t{f}\in\CCCC_d$ such that each $f_n$ is  quasiconformal conjugate to $\t{f}$. Moreover, there exists a continuous surjection $\phi:J_{\t{f}}\to J_f$ satisfying $$f\circ\phi=\phi\circ \t{f}\ \text{on}\ J_{\t{f}},$$ such that   $\t{z}=\phi^{-1}(z)$ is a singleton and $\pi_f^{-1}(z)=\pi_{\t{f}}^{-1}(\t{z})$ unless $f^k(z) = c^*$ for some $k \ge 0$, in which case $\phi^{-1}(z)$ consists of $d_z=\deg(f^{k+1}, z)$ distinct points $\t{z}_1,\ldots,\t{z}_{d_z}$ satisfying $$\pi_f^{-1}(z)=\bigsqcup_{j=1}^{d_z}\pi_{\t{f}}^{-1}(\t{z}_j).$$
					\end{enumerate}
	\end{theorem}
	
	We refer to each map $f_n$ in this theorem as a \emph{dynamical perturbation} of $(f,c^*)$.\vspace{3pt}
	
	The construction of this perturbation proceeds in three stages. First, we apply a topological surgery to $f$ to obtain a sequence of topological polynomials $F_n$ converging to $f$, under which the orbit of the critical point $c^*$
	converges to an attracting cycle. Next, we demonstrate that each $F_n$ is equivalent to a genuine  polynomial $f_n$. Finally, we prove that $f_n\to f$ as $n\to\infty$.
	
	A related approach was used in \cite{CT3,G,GZ}. The first and third stages have already been established for general geometrically finite rational maps \cite{G}. The second stage was verified in the case where $f$ is a subhyperbolic polynomial and $c^*$ is perturbed into the basin of infinity \cite{CT3, GZ}, as well as when $f$ is a subhyperbolic renormalizable Newton map and $c^*$
	is perturbed into the basins of super-attracting fixed points \cite{G}.
	
	In the proof of Theorem \ref{thm:perturbation}, we verify  the second stage for the case when $c^*$
	is perturbed into a bounded Fatou domain. The additional novelty of Theorem \ref{thm:perturbation} lies in showing that the resulting perturbed polynomials $f_n$ satisfy the  dynamical conditions (4) and (5). This theorem provides a hyperbolic–subhyperbolic deformation of polynomials, complementing the hyperbolic-parabolic deformation \cite{CT2} introduced in the next section.
	
\subsection{Proof of Theorem \ref{thm:perturbation}}
		The argument is organized into five steps. \vspace{2pt}

We select a disk within each Fatou domain of $f$ that contains critical or postcritical points, so that their union $\mathcal{A}$ is $f$-invariant, i.e., $f(\mathcal{A})\subseteq \mathcal{A}$, satisfying
\begin{itemize}
\item  $(C_f \cup P_f) \cap F_f\subseteq \AAA$,  
\item $\partial\AAA\cap J_f$ consists of all parabolic periodic points of $f$, and
\item $f^k|_\AAA$ uniformly converges to attracting or parabolic cycles as $k\to\infty$.
\end{itemize}
After suitably adjusting arcs of a Hubbard tree $T_f$ inside the Fatou set, we may assume that for every component $A$ of $\mathcal{A}$, the intersection $A \cap T_f$ is non-empty and connected, and that $T_f^* = T_f \cup \mathcal{A}$ is $f$-invariant. \vspace{5pt}

{\bf 1. Topological surgery.}
		\vspace{5pt}
Set $v^*=f(c^*)$. For each $n\geq1$, choose a disk $D_{n}(v^*)\ni v^*$ with diameter less than $1/n$. Since $K_f$ is full and locally connected, we may assume the disk satisfies that $\overline{D_n(v^*)}\cap K_f$ is full and any bounded Fatou domain $U$ with $v^*\notin \partial U$ is either contained in $\overline{D_n(v^*)}$ or intersects $\overline{D_n(v^*)}$ in at most one point.
This implies, when $v^*$ is disjoint from the boundary of any bounded Fatou domain, $\overline{D_n(v^*)}\cap T$ is a tree for any tree $T\subseteq K_f$.

Let $D_n(c^*)$ be the component of $f^{-1}(D_n(v^*))$ containing $c^*$. By passing to a subsequence if necessary, we may assume further that  $D_{n+1}(c^*)\Subset D_n(c^*)$ and $\overline{D_n(c^*)}\cap\overline{\AAA}=\emptyset$ for all $n$. Then $\overline{D_n(c^*)}\cap T_f^*=\overline{D_n(c^*)}\cap T_f$ is a tree.\vspace{1pt}
		
	Now we fix an $n\geq1$ and define a perturbation $F_n$ of $f_n$ by a topological surgery.\vspace{1pt}
	
	Note first that there exists $k_n\geq0$ such that $f^{-k_n}(T_f)\cap D_n(v^*)$ contains a point in the Fatou set. Therefore, by replacing $T_f$ with $f^{-k_n}(T_f)$ if necessary, we may assume that $T_f\cap D_n(v^*)$ intersect the Fatou set.
	
	Next, there exists a minimal $r \geq 1$ such that every $z \in f^{-r}(c^*) \cap T_f^*$ is neither a critical point nor a postcritical point of $f$ and $T_f^*$ is locally an arc near $z$.
	For simplicity, we may assume $r=1$. Let $a_1, \ldots, a_m$ be the preimages of $c^*$ in $T_f^*$, and denote by $D_n(a_j)$ the component of $f^{-1}(D_n(c^*))$ containing $a_j$. Then, $f:D_n(a_j)\to D_n(c^*)$ is a homeomorphism and $$T_{D_n(a_j)}:=\overline{D_n(a_j)}\cap T_f^*$$ is a tree.
		
\begin{figure}[http]
	\centering
	\begin{tikzpicture}
		\node at (0,0){ \includegraphics[width=16cm]{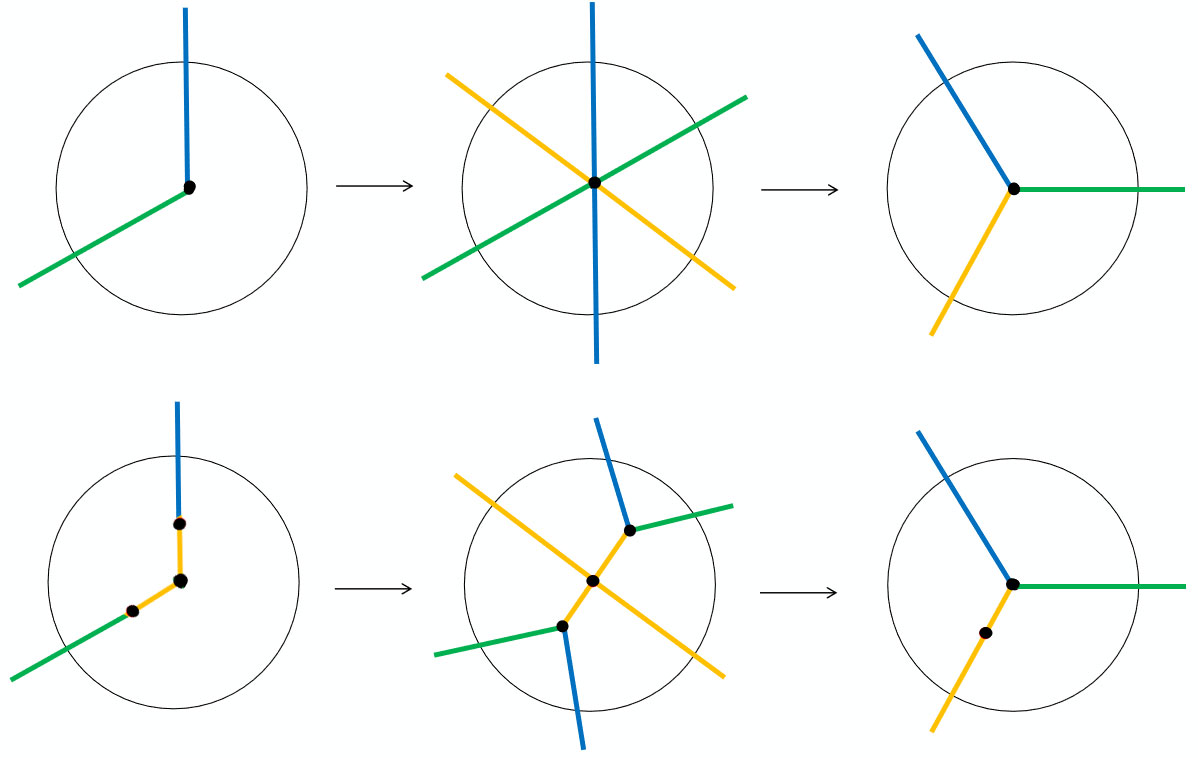}};
		\node at (-3, 3) {$f$};
		\node at (2.75, 3) {$f$};
		\node at (2.75, -2.5){$F_n$};
		\node at (-3,-2.5){$F_n$};
		\node at (-5.1, 2.5){$a_j$};
		\node at (0.4, 2.58){$c^*$};
		\node at (0.4, -2.7){$c^*$};
		\node at (5.9, 2.8){$v^*$};
		\node at (5.92, -3.5){$F_n(c^*)$};
	\end{tikzpicture}
	\caption{The construction of $F_n$ on $D_n(c^*)$ and $D_n(a_j)$.}\label{fig:Fn}
\end{figure}
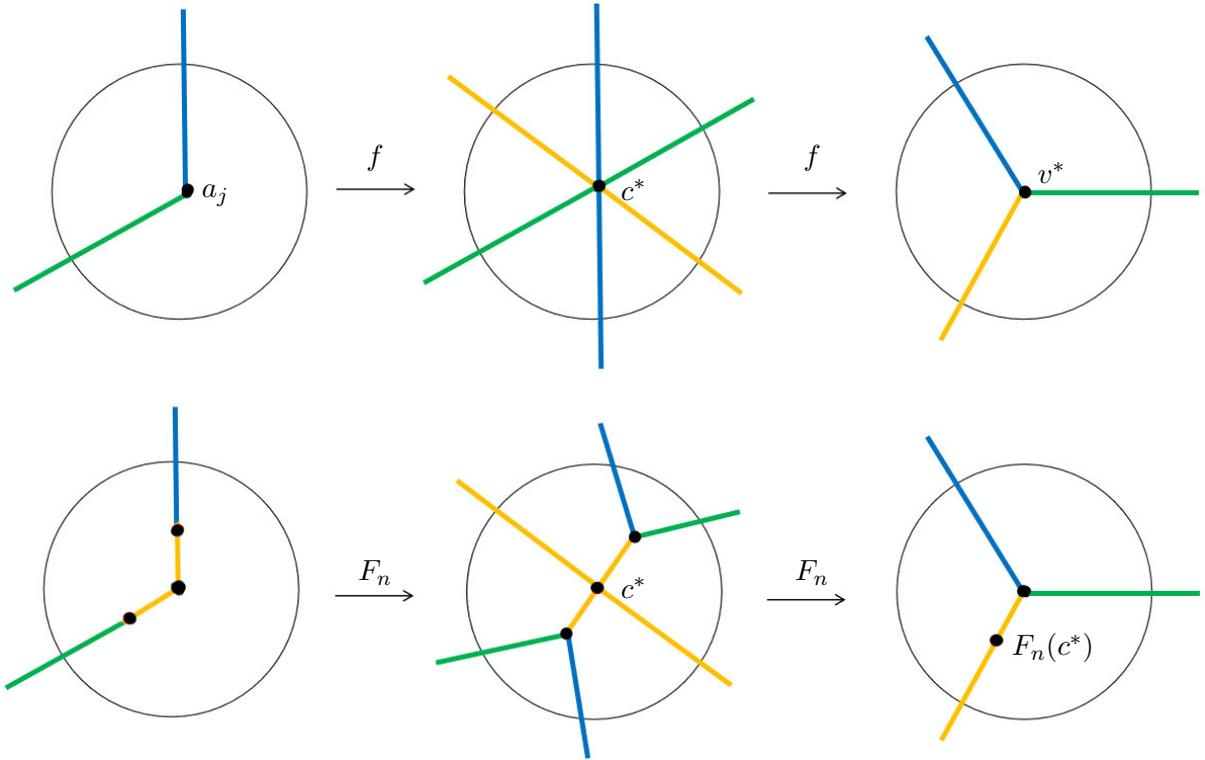			
	The perturbation $F_n:\C\to \C$ of $f_n$ is a branched covering defined as follows (see Figure \ref{fig:Fn}):
		\begin{itemize}
			\item $F_n(z) = f(z)$ for $z \notin D_n(c^*) \cup \bigcup_{j=1}^m D_n(a_j)$;\vspace{1pt}
			\item $F_n: D_n(c^*) \to D_n(v^*)$ is a branched covering of degree $\deg(f,c^*)$ with the unique critical point $c^*$,  such that $v_n^*:=F_n(c^*) \in T_f \cap {F}_f$ and $f^k(v_n^*)\not\in D_n(c^*)$ for all $k\geq1$. Define $$T_{D_n(c^*)} = (F_n|_{\overline{D_n(c^*)}})^{-1}(T_f \cap \overline{D_n(v^*)}).$$
			It then follows that $T_{D_n(c^*)}$ is a tree containing $\partial D_n(c^*)\cap T_f^*$.
			\item For each $j$, the restriction $F_n: D_n(a_j) \to D_n(c^*)$ is a homeomorphism satisfying $F_n(a_j) = c^*$, and $F_n(T_{D_n(a_j)})$ is a subtree of $T_{D_n(c^*)}$.
		\end{itemize}
				
		By construction,   $C_{f}=C_{F_n}$. Define
		\begin{equation}\label{eq:22}
			P_n=P_{F_n}\cup \bigcup_{k\geq0} f^k(v^*) \quad\text{and}\quad T_{F_n}^*=(T_f^*\setminus D_n(c^*))\cup T_{D_n(c^*)}.
		\end{equation}
		Then both $P_n$ and $\AAA$ are contained in $T_{F_n}^*$, and all the three sets are $F_n$-invariant.  
		
		Thus $F_n$ is  \emph{semi-rational}: it is holomorphic near the infinity and  each cycle in $P_{F_n}$, which is either attracting or parabolic, and any attracting petal at a parabolic point of $F_n$ contains points of $P_{F_n}$. This definition of  semi-rational maps is slightly stronger than that in \cite{CT2}.
		
		\vspace{5pt}
		
		{\bf 2. Rational realization}\vspace{5pt}
		
		A \emph{marked semi-rational polynomial} is a pair $(F,P)$ where $F$ is a semi-rational polynomial and $P$ is a closed set containing $P_F$ with $F(P) \subseteq P$ and $\#(P \setminus P_F) < \infty$. Thus, the pair $(F_n, P_n)$ constructed in Step 1 is a marked semi-rational polynomial.
		
		Two marked semi-rational polynomials $(g_1,P_1)$ and $(g_2,P_2)$ are \emph{c-equivalent} if there exist homeomorphisms $\phi, \psi: \mathbb{C} \to \mathbb{C}$ and a neighborhood $W$ of the union of all attracting cycles and parabolic cycles of $g_1$ (including the infinity) such that:
		\begin{itemize}
			\item $\phi \circ g_1 = g_2 \circ \psi$ on $\mathbb{C}$;
			\item $\phi$ is holomorphic on $W$;
			\item $\phi$ and $\psi$ are isotopic rel $P_1 \cup \overline{W}$.
		\end{itemize}
	When $P_1 = P_{g_1}$ (so $P_2 = P_{g_2}$), we say $g_1$ is \emph{c-equivalent} to $g_2$.

The following theorem provides a criterion for determining when a marked semi-rational polynomial is c-equivalent to a marked polynomial.
	
A Jordan curve  $\alpha\subseteq \C\setminus P$ is called \emph{essential} if each component of $\C\setminus \alpha$ contains at least two points in $P\cup\{\infty\}$. An open arc $\beta\subseteq\C\setminus P$ is called \emph{essential} if its endpoints $z_0,z_1$ belong to $P$ and either $z_0\not=z_1$, or $z_0=z_1$ and both components of $\C\setminus \overline{\beta}$ intersects $P$.
A \emph{multicurve} in $\C\setminus P$ is  a finite collection $\G$ of pairwise disjoint, pairwise non-isotopic, essential Jordan curves in $\C\setminus P$.
		
		Let $(F,P)$ be a semi-rational polynomial. A multicurve $\{\g_1,\ldots,\g_p\}$ in $\C\setminus P$ is called a
		\emph{Levy cycle} of $(F,P)$ if for each $\g_i$,
		there exists a Jordan curve $\wt{\g}_i$  isotopic to $\g_i$ in $\C\setminus P$, such that $\wt{\g}_i$ is a component of $F^{-1}(\g_{i+1})$ and  $F$ maps $ \wt{\g}_i$ homeomorphically onto $\g_{i+1}$, where the indices are taken modulo $p$. We call $\{\wt{\g}_1,\ldots,\wt{\g}_p\}$  the \emph{adjacent multicurve} of the Levy cycle.
		
An essential open arc $\beta\subseteq\C\setminus P$ is called a \emph{connecting arc} of $(F,P)$ if $\beta$ is isotopic rel $P$ to a component of $F^{-p}(\beta)$ for some integer $p\geq1$, $\beta$ avoids attracting petals of all parabolic points, and the endpoints of $\beta$ are not attracting.
		
		\begin{theorem}[Thurston-Cui-Tan\cite{DH3,CT2}]\label{thm:cui-tan-Thurston}
			Let $(F,P)$ be a marked semi-rational polynomial. Then $(F,P)$ is c-equivalent to a marked polynomial if and only if it has neither Levy cycles nor connecting arcs. Moreover, such a polynomial is unique up to affine conjugacy.
		\end{theorem}
		
We remark that if the endpoints of a connecting arc $\beta$ of $(F,P)$ contain no parabolic points, then $\beta$ induces a Levy cycle of $(F,P)$.	
		
\begin{lemma}\label{pro:c-equivalent}
Let $Q\subseteq J_f\setminus P_f$ be a finite set such that $f(Q)\subseteq Q$. Then $(F_n,P_{n}\cup Q) $ has neither Levy cycles nor connecting arcs for each sufficiently large $n$.
\end{lemma}
\begin{proof}
By construction, the set $Q_n=P_n\cup Q$ is  a marked set of $F_n$ for all sufficiently large $n$. Fix such an $n$. Since the $F_n$-orbit of $c^*$ eventually enters $\AAA$ and $F_n(\AAA)\subseteq \AAA$, there exists a minimal integer $r_n\geq0$ such that $\AAA_n:=F_n^{-r_n}(\AAA)$ contains every critical points of $F_n$ iterated to $c^*$.

Suppose, contrary to the lemma, that $(F_n,Q_n)$ has either a Levy cycle $\{\gamma_1,\ldots,\gamma_p\}$ with adjacent multicurve $\{\widetilde{\gamma}_1,\ldots,\widetilde{\gamma}_p\}$, or a connecting arc $\beta$ whose endpoints contains parabolic points.
			
Since $F_n^k(\AAA_n)$ uniformly converges to attracting or parabolic cycles as $k\to\infty$, we may adjust each $\gamma_i$ (resp.\,$\beta$) within its isotopy class rel $Q_n$ so that $\gamma_i$ (resp.\,$\beta$) is disjoint from $F_n^{k_0+1}(\mathcal{A}_n)$, while preserves  minimal intersection  with $F_n^k(\mathcal{A}_n)$ for all $0\leq k\leq k_0$.\vspace{8pt}
			
			{\bf Claim 1}. \emph{For each $i\in\{1,\ldots,p\}$, the bounded component $D_i$ of $\C\setminus \g_i$ avoids $\AAA_n$. }
			
			\begin{proof}[Proof of Claim 1]
				We first show that $\gamma_i$ is disjoint from $\mathcal{A}_n$. Assume for contradiction that $\gamma_i \,\cap \mathcal{A}_n \neq \emptyset$. By minimality of the intersection, it follows that $\widetilde{\gamma}_i \cap \mathcal{A}_n \neq \emptyset$, and hence $\gamma_{i+1} \cap F_n(\mathcal{A}_n) \neq \emptyset$. Again, by minimality of the intersection of $\gamma_{i+1}$ with $F_n(\mathcal{A}_n)$, we deduce $\gamma_{i+2} \cap F_n^2(\mathcal{A}_n) \neq \emptyset$. Continuing this process, we eventually find some $\gamma_j$ intersecting $F_n^{k_0+1}(\mathcal{A}_n)$, contradicting the assumption on $\{\gamma_1, \dots, \gamma_p\}$.
				
				Since \(F_n(\mathcal{A}_n) \subseteq \mathcal{A}_n\), it follows that \(\widetilde{\gamma}_1, \ldots, \widetilde{\gamma}_p\) are disjoint from \(\mathcal{A}_n\). For each \(i\), let \(\widetilde{D}_i\) denote the bounded component of \(\mathbb{C}\setminus\widetilde{\gamma}_i\). Then
				\begin{equation}\label{eq:45}
					F_n: \widetilde{D}_j \to D_{j+1} \quad \text{is a homeomorphism for all } j \geq 1.
				\end{equation}
				
				Now suppose, contrary to the claim, that some \(D_i\) intersects a component \(A\) of \(\mathcal{A}_n\). The preceding argument shows \(A \subseteq D_i\), hence \(A \subseteq \widetilde{D}_i\). For each \(k\), let \(A_k\) be the component of \(\mathcal{A}_n\) containing \(F_n^k(A)\). Then \(D_{i+1}\) contains \(A_1\), and thus \(A_1 \subseteq \widetilde{D}_{i+1}\). By induction, \(\widetilde{D}_{i+k} \supseteq A_k\) for all \(k \geq 0\). On the other hand, some \(A_k\) contains critical points of \(F_n\), contradicting \eqref{eq:45}. This completes the proof of Claim 1.
			\end{proof}
			
A similar argument also shows that \vspace{8pt}
			
			{\bf Claim 2}. \emph{The connecting arc $\beta$ is disjoint from $\AAA_n$. }	\vspace{8pt}

			Now we define a postcritically finite semi-rational polynomial $G_n$, satisfying
			\begin{itemize}
				\item $G_n(z)=F_n(z)$ for $z\in\C\setminus \AAA_n$;
				\item each component of $\AAA_n$ contains at most one critical or postcritical point of $G_n$;
				\item $G_n$ has an invariant tree $T_{G_n}\supseteq P_{G_n}$, which coincides with $T_{F_n}^*$ outside of $\AAA_n$.
			\end{itemize}
			
			By the preceding claims and the construction of $G_n$, either the multicurve $\{\gamma_1,\ldots,\gamma_p\}$ is  still a Levy cycle for the marked semi-rational polynomial $(G_n, Q_n^*)$, where $Q_n^* = (Q_n \setminus \mathcal{A}_n) \cup P_{G_n}$, or $\beta$ is a connecting arc for $(G_n,Q_n^*)$. Furthermore, since $G_n$ has no parabolic cycles, the connecting arc $\beta$ also induces a Levy cycle of $(G_n,Q_n^*)$ by the remark before this lemma.
			
			On the other hand, by the construction of $T_{F_n}^*$ in \eqref{eq:22},
			one may verify directly that, if an edge of $T_{G_n}$ has endpoints $v, v'$ not iterated into $\mathcal{A}_n$, then some iterate $G_n^j$ maps $v$ and $v'$ to points that are not endpoints of any edge in $T_{G_n}$. Hence, $(G_n, Q_n^*)$ is c-equivalent to a marked polynomial by \cite[Theorem 1.1]{Poi}. By Theorem \ref{thm:cui-tan-Thurston}, $(G_n, Q_n^*)$ admits no Levy cycle, a contradiction.
			This completes the proof of the lemma.
		\end{proof}

		By Lemma~\ref{pro:c-equivalent} and Theorem~\ref{thm:cui-tan-Thurston}, for all sufficiently large $n$, there exists a unique polynomial $f_n \in \PPPP_d$ such that $F_n$ is $c$-equivalent to $f_n$. By the Rees-Shishikura Theorem~\cite{CPT}, the map $f_n$ satisfies Statements (1)--(3) of the theorem.
		
		\begin{proposition}\label{pro:convergence}
			The geometrically finite map $f_n$ uniformly converges to $f$ as $n\to\infty$.
		\end{proposition}
		
		The proof of this convergence result was previously given in \cite[Proposition 3.3]{G} as well as \cite[Step IV of the proof of Theorem 1.1]{GZ}. Therefore, it remains to verify Statements (4) and (5).\vspace{5pt}
		
		{\bf 3. Correspondence of maximal Fatou chains.}
		\vspace{5pt}
		
		Recall that $\EEE_f$ denotes the collection of all maximal Fatou chains of $f$ containing a critical or postcritical point.
		Then
		\begin{equation}\label{eq:5555}
		\mathcal{H}_f := T_f^* \cap \bigcup_{E\in\mathcal{E}_f} E
		\end{equation}
		is an $f$-invariant (decorated) forest containing all critical  points in maximal Fatou chains. By Propositions \ref{pro:iterate} and \ref{pro:cluster}, the set $Q$ of  periodic points in $\HHH_f$ is  finite, and \vspace{5pt}
		
		 {\bf Claim 3}. \emph{The $f$-orbit of every point in $\HHH_f$ eventually enters $\AAA\cup Q$.}
		\vspace{5pt}

		Since $c^* \in \HHH_f$ by the condition in Statement (4), we may require $F_n(c^*)\in \HHH_f\cap F_f$. The modification from $T_f^*$ to $T^*_{F_n}$ in Step 1, when restricted to $\HHH_f$, yields an $F_n$-invariant forest $\HHH_{F_n}$ that differs from $\HHH_f$ only within $D_n(c^*)$.

		\vspace{5pt}
		
		 {\bf Claim 4}. \emph{For any preperiodic point $z\in \HHH_f\cap J_f$, there exists $\epsilon>0$ such that the  punctured $\epsilon$-neighborhood of $z$ is disjoint from $\HHH_f\cap \bigcup_{k\geq0}f^{-k}(z)$.}
		\begin{proof}[Proof of Claim 4]
			It suffices to consider the case that $z$ is periodic. Let $D_z$ be an open neighborhood of $z$ such that each component of $(\HHH_f\setminus\{z\})\cap D_z$ is either a disk in a parabolic domain, or an open arc, called a \emph{local edge} to $z$. If such a component  lies in $F_f$,  it clearly avoids $\bigcup_{k\geq0}f^{-k}(z)$.
			
		Now, let $\g:(0,1)\to\C$ be a local edge to $z$ with $\lim_{t\to0}\g(t)=z$ and $\g(0,\epsilon)\not\subseteq F_f$ for each sufficiently small $\epsilon>0$.
		Then
			 there exist $m \geq 1$, a level-$m$  Fatou chain $K$, and $\epsilon > 0$ such that:
			\begin{enumerate}
				\item $\g(0,\epsilon) \subseteq K$;
				\item $z \notin \bigcup_{k \geq 0} E_k$, where $E_0\subseteq K$ is the unique level-$(m-1)$  Fatou chain and $K=\overline{\bigcup_{k\geq0} E_k}$;
				\item every point of $\g(0,\epsilon)$ is mapped into $E_0$ under iteration.
			\end{enumerate}
			Hence, the iterated preimages of $z$ is disjoint from  $\g(0,\epsilon)$. 
			
			Since $\mathcal{H}_f$ has finitely many local edges to $z$, the claim 4 follows.
		\end{proof}
		Recall the assumptions adding on $\{a_1,\ldots,a_m\} = f^{-1}(c^*) \cap T_f^*$ in Step 1. By taking preimages, we may assume $a_1,\ldots,a_m\in  \mathcal{H}_{F_n}$. Define $H_n := D_n(c^*) \cap \mathcal{H}_{F_n}$.\vspace{5pt}
		
		 {\bf Claim 5}. \emph{For all sufficiently large $n$, the $F_n$-orbit of any point in $H_n$ avoids $D_n(c^*)$.}

		\begin{proof}[Proof of Claim 5]
		 
			\vspace{1pt}
			
			Choose $\epsilon, \epsilon' > 0$ such that Claim 4 holds for $c^*$ and $\epsilon$, and $\D_{\epsilon'}(v^*) \subseteq f(\D_\epsilon(c^*))$. By expansivity \eqref{eq:expanding}, there exists $\epsilon_0 < \epsilon'/2$ such that for any arc $\gamma \subseteq K_f \setminus \mathcal{A}$ with $\mathrm{diam}(\gamma) < \epsilon_0$, the diameters of components of $f^{-k}(\gamma)$ are less than $\epsilon_0$ for all $k \geq 0$.

				Suppose the claim does not hold. Then for some large $n$ with $D_n(v^*)\subseteq \D_{\epsilon_0}(v^*)$, there exist $w \in H_n$ and $k_0 > 1$ such that $F_n^{k_0}(w) \in H_n$. As a result, there exist $j \in \{1,\ldots,m\}$ and $z \in D_n(a_j) \cap \mathcal{H}_{F_n}=:\gamma_j$ with $F_n(z) = w$, such that $F_n^{k_0}(z) \in \gamma_j$.
				
				Let $I_{k_0} \subseteq \gamma_j$ be the arc from $a_j$ to $F_n^{k_0}(z)$. For $0 \leq k < k_0$, let $I_k$ be the lift of $I_{k_0}$ under $F_n^{k_0-k}$ based at $z_k := F_n^k(z)$. The other endpoint $y_k$ of $I_k$ belongs to $F_n^{-(k_0-k)}(a_j)$. Since $D_n(a_j) \cap P_{F_n} = \emptyset$,  $F_n$ is injective near each $I_k$, and as $\gamma_j$ contains no branch points of $T_{F_n}^*$, we have $I_k \subseteq \mathcal{H}_{F_n}$. 
				
				Since $F_n|_{I_k} = f|_{I_k}$ for $2 \leq k < k_0$, the arc $I_2$ is also a component of $f^{-(k_0-2)}(I_{k_0})$. Note that $z_2$ lies in the $\epsilon_0$-neighborhood of $v^*$. By choice of $\epsilon_0$, it follows that $I_2 \subseteq \D_\epsilon(v^*)$. Thus $\D_\epsilon(v^*) \cap \mathcal{H}_f$ contains a point $y_2 \in f^{-(k_0-2)}(a_j) \subseteq f^{-(k_0-1)}(c^*)$. As $\D_{\epsilon'}(v^*) \subseteq f(\D_\epsilon(c^*))$, the set $\D_\epsilon(c^*) \cap \mathcal{H}_f$ contains a point of $f^{-k_0}(c^*)$, contradicting Claim 4 by our choice of $\epsilon$. This proves  Claim 5.
		\end{proof}
		
As a corollary of Claim 5, for all sufficiently large \(n\), every \(F_n\)-orbit in \(\mathcal{H}_{F_n}\) eventually enters \(\mathcal{A}\cup Q\).  
Indeed, take any \(z\in\mathcal{H}_{F_n}\).  
\begin{itemize}
\item If the \(F_n\)-orbit of \(z\) avoids \(D_n(c^*)\), then \(z\in\mathcal{H}_f\) and its \(F_n\)-orbit coincides with its \(f\)-orbit; thus the claim follows directly from Claim 3.  
\item If \(F_n^k(z)\in D_n(c^*)\) for some \(k\ge0\), then by Claim 5 the orbit of \(F_n^{k+1}(z)\) never returns to \(D_n(c^*)\), and the conclusion again holds.
\end{itemize}
		
Fix a sufficiently large $n$. Recall that $Q$ is the finite set of all  periodic points in $\HHH_f$. By Lemma~\ref{pro:c-equivalent} and Theorem~\ref{thm:cui-tan-Thurston}, there exist a polynomial $f_n \in \PPPP_d$ and homeomorphisms $\psi_0, \psi_1: (\mathbb{C}, P_n \cup Q) \to (\mathbb{C}, P_n \cup Q)$ that are holomorphic in a neighborhood $W$ of all attracting or parabolic cycles of $F_n$ and the infinity, such that $\psi_0$ and $\psi_1$ are isotopic rel $W \cup P_n \cup Q$, and  $\psi_0 \circ F_n = f_n \circ \psi_1$ on $\mathbb{C}$.
		By the Homotopic Lifting Lemma, we obtain a sequence $\{\psi_k\}_{k \geq 0}$ satisfying
		\[
		\psi_k \circ F_n = f_n \circ \psi_{k+1} \text{ on } \mathbb{C},
		\]
		where each $\psi_k$ is isotopic to $\psi_{k+1}$ rel $F_n^{-k}(W \cup P_n \cup Q)$. By \cite[Lemma 2.2]{CPT}, these maps converge to a quotient map $\psi: \mathbb{C} \to \mathbb{C}$ such that
		\[
		\psi \circ F_n = f_n \circ \psi  \text{ on } \mathbb{C},
		\]
		and $\psi$ is injective on $\bigcup_{k \geq 0} F_n^{-k}(W \cup P_n \cup Q)$.
		
		By the corollary of Claim 5,  $Q$ coincides with the set of periodic points of $F_n$ within $\HHH_{F_n}$, and $$\mathcal{H}_{F_n} \subseteq \bigcup_{k\geq 0} F_n^{-k}(W \cup P_n \cup Q).$$
		Thus, $\psi$ conjugates $F_n|_{\mathcal{H}_{F_n}}$ to $f_n|_{\mathcal{H}_{f_n}}$ with $\mathcal{H}_{f_n} := \psi(\mathcal{H}_{F_n})$.
		In particular, $\mathcal{H}_{f_n}$ contains $N=\# Q$ periodic points. It is immediate that $\psi$ induces a bijection
		$$\varsigma_n: \textup{Comp}(\mathcal{H}_{f}) \to \textup{Comp}(\mathcal{H}_{f_n})$$ such that $c\in C_f$ lies in a component $H $ of $\HHH_f$ if and only if $c_n$ belongs to $\varsigma_n(H)$.
		
		There is a canonical bijection \(\varsigma: \textup{Comp}(\mathcal{H}_f) \to \mathcal{E}_f\) sending each component of \(\mathcal{H}_f\) to the maximal Fatou chain containing it.
		
		Since \(\mathcal{H}_{f_n}\) contains only finitely many periodic points, Proposition~\ref{pro:cluster}\,(4) implies that every component of \(\mathcal{H}_{f_n}\) belongs to some maximal Fatou chain of \(f_n\) within $\mathcal{E}_{f_n}$. This gives an inclusion \(\widetilde{\varsigma}_n: \textup{Comp}(\mathcal{H}_{f_n}) \to \mathcal{E}_{f_n}\).
		
		To show that \(\widetilde{\varsigma}_n\) is bijective, we first claim: for any two critical points \(c_n, c_n'\) of \(f_n\) in distinct components of \(\mathcal{H}_{f_n}\), there exists a pre-repelling point \(z_n\) of period greater than \(N\) such that \(c_n\) and \(c_n'\) lie in distinct components of \(K_{f_n}\setminus\{z_n\}\).
		
		Indeed, since \(c\) and \(c'\) are not in the same maximal Fatou chain of \(f\), by Lemma~\ref{lem:non-cluster} we can choose rational angles \(\theta,\eta\) with periods greater than \(N\) such that \(\pi_f(\theta)=\pi_f(\eta)=z\), the orbit of \(z\) avoids the critical points of \(f\), and the arc \(\overline{R_f(\theta)}\cup\overline{R_f(\eta)}\) separates \(c\) from \(c'\). Using Lemma~\ref{lem:stable}\,(2) and taking \(n\) large enough, we have \(\pi_{f_n}(\theta)=\pi_{f_n}(\eta)=z_n\) and the arc \(\overline{R_{f_n}(\theta)}\cup\overline{R_{f_n}(\eta)}\) separates \(c_n\) from \(c_n'\). This completes the proof of the claim.

		Let $H$ be a periodic element of $\mathcal{H}_{f_n}$ with period $p$. Define $H^k$ as the component of $f_n^{-pk}(H)$ containing $H$, and set $E_{f_n,H} = \overline{\bigcup_{k\geq 0} H^k}$. Since $H$ contains a critical point and all points in $H$ have period at most $N$, the preceding claim implies that each $H^k$, and hence $E_{f_n,H}$, is separated from the critical points of $f_n$ outside $H$ by pre-repelling points.
		Consequently, $$E_{f_n,H}\cap E_{f_n,H'}=\emptyset$$ for any other periodic component \(H'\) of \(\mathcal{H}_{f_n}\). Hence \(E_{f_n,H}\) is a maximal Fatou chain, and \(\widetilde{\varsigma}_n\) is injective on the periodic elements of \(\mathcal{H}_{f_n}\).
		
		A further similar argument shows that $\widetilde{\varsigma}_n$ is bijective. Then the map $\iota_n := \widetilde{\varsigma}_n \circ \varsigma_n \circ \varsigma^{-1}$ satisfies the requirements of Statement (4).

		\vspace{5pt}
		
		{\bf 4. Quasiconformal deformation}\vspace{5pt}
		
		For the proof of Statement (5), we now assume that $c^*$ lies on the boundary of a bounded Fatou domain $U$ of $f$ and begin by establishing the following claim.\vspace{5pt}
		
		{\bf Claim 6}. \emph{For each sufficiently large  $n$, there exist homeomorphisms $\xi_0,\xi_1:\C\to\C$ and a neighborhood $W_0$ of the union of all attracting or parabolic cycles  of $F_n$ (including the infinity), such that
			$\xi_0\circ F_n=F_{n+1}\circ \xi_1$ on $\C$, $\xi_0$ is quasiconformal on $\C$ and is identical near $\infty$,  and $\xi_0,\xi_1$ are isotopic relative to $P_{n}\cup \overline{W_0}$.}
		
		\begin{proof}[Proof of claim 6]
			Since $c^*\in\partial U$, we can define $F_n$ and $F_{n+1}$ in Step 1 such that $F_n(c^*)$ and $F_{n+1}(c^*)$ both lie in $D_n(v^*)\cap f( U)$, and that there exists a disk $B_0\Subset D_n(v^*)\cap f( U)$ with smooth boundary satisfying the following properties:
			\begin{itemize}
				\item $B_0$ contains both $F_n(c^*)$ and $F_{n+1}(c^*)$;
				\item for each $k\geq 1$, $f^k:B_0\to B_k:=f^k(B_0)$ is a homeomorphism and $\overline{B_k}\cap P_f=\emptyset$;
				\item the closed disks $\overline{B_k},k\geq1,$ are pairwise disjoint.
			\end{itemize}

			Let $h_0:\overline{B_0}\to \overline{B_0}$ be a quasiconformal map such that $h_0=id$ on $\partial B_0$ and $h_0(F_n(c^*))=F_{n+1}(c^*)$. Then, for each $k\geq1$, define a quasiconformal map $h_k:\overline{B_k}
			\to \overline{B_k}$ by $$h_k=f^k\circ h_0\circ (f^k|_{B_0})^{-1}.$$ Since $f^k$ coincides with both $F_n^k$ and $F_{n+1}^k$ on $\overline{B_0}$, it follows that
			$h_k=id$ on $\partial B_k$ and $h_k(F_n^{k+1}(c^*))=F_{n+1}^{k+1}(c^*)$. Set $\mathcal{B}=\bigcup_{k\geq0}B_k$.

			We define a map $\xi_0:\C\to \C$ by
			\[\xi_0(z)=\left\{
			\begin{array}{ll}
				z, & \hbox{if $z\in\C\setminus \mathcal{B}$;} \\[2pt]
				h_k(z), & \hbox{if $z\in B_k$ for some $k\geq 0$.}
			\end{array}
			\right.\]
			It is straightforward to verify that $\xi_0$ is quasiconformal on $\C$, and $\xi_0(P_{n})=P_{n+1}$. Moreover, since $F_n$ coincides with  $F_{n+1}$  outside of $D_n(c^*)$ and $D_n(a_j)$ with $j=1,\ldots,m$, we conclude that
			\begin{equation}\label{eq:33}
				\xi_0\circ F_n(z)=F_{n+1}\circ \xi_0(z)\ \textup{if } z\not\in \left(f^{-1}(\mathcal{B})\setminus\mathcal{B}\right)\cup D_n(c^*)\cup\bigcup_{j=1}^m D_n(a_j).
			\end{equation}
			
			To construct $\xi_1:\C\to\C$, we modify $\xi_0$ on three types of disks: $D_n(c^*)$, $D_n(a_j)$, and each component $B$ of $f^{-1}(\mathcal{B})\setminus \mathcal{B}$ not contained in $D_n(c^*)$.\vspace{1pt}
			
			First, note that $F_n=F_{n+1}=f$ on such a component $B$
			and $f:B\to B_k$ is conformal for some $k$. Then the quasiconformal map $h_B: \overline{B}\to \overline{B}$ is defined by $h_B=(f|_B)^{-1}\circ h_k\circ f|_B$. It follows that $h_B=id$ on $\partial B$.
			
			Next, the maps $F_n,F_{n+1}:D_n(a_j)\to D_n(c^*)$ are both homeomorphisms and send $a_j$ to $c^*$. Define the homeomorphism $h_{a_j}:D_n(a_j)\to D_n(a_j)$ by $h_{a_j}=(F_{n+1}|_{D_n(a_j)})^{-1}\circ F_n$ for $j=1,\ldots, m$.
			
			Finally, the disk $D_n(c^*)$ contains the unique critical point $c^*$ of both $F_n$ and $F_{n+1}$, and the punctured disk $D^*_n(c^*)$ is disjoint from $P_n$ and $P_{n+1}$. Thus, the maps
			\[\text{ $F_n:D^*_n(c^*)\to D^*_n(v^*)$  and $F_{n+1}:D^*_n(c^*)\to D^*_n(v^*)$}\]
			are both coverings of the same degree.
			
			Note that $F_n=F_{n+1}=f$ on $\partial D_n(c^*)$ and $\xi_0=id$ on $\partial D_n(v^*)$. By the Covering Lifting Lemma, there exists a unique homeomorphism $\tilde{h}:\overline{D_n(c^*)}\to \overline{D_n(c^*)}$ such that $\xi_0\circ F_n=F_{n+1}\circ \tilde{h}$ on $\overline{D_n(c^*)}$ and $\tilde{h}=id$ on $\partial D_n(c^*)$.
			
			We now define the map $\xi_1:\C\to\C$ piecewise as follows:
			\[\xi_1(z)=\left\{
			\begin{array}{ll}
				\xi_0(z), & \hbox{if $z\not\in\left( f^{-1}(\mathcal{B})\setminus\mathcal{B}\right)\cup D_n(c^*)\cup(\bigcup_{i=1}^m D_n(a_j))$;} \\[2pt]
				h_B(z), & \hbox{if $z\in B$, a component of $f^{-1}(\mathcal{B})\setminus\mathcal{B}$ not contained in $D_n(c^*)$;}\\[2pt]
				\tilde{h}(z), &\hbox{if $z\in D_n(c^*)$;}\\[2pt]
				h_{a_j}(z), &\hbox{if $z\in D_n(a_j)$ for some $j\in\{1,\ldots,m\}$.}
				
			\end{array}
			\right.\]
			The constructions of $h_B$, $\tilde{h}$ and $h_{a_j}$ ensure that $\xi_0$ and $\xi_1$ satisfy all requirements in Claim 6.
		\end{proof}
		
		Since $F_n$ and $F_{n+1}$ are c-equivalent to $f_n$ and $f_{n+1}$, respectively, Claim 6 implies that there are homeomorphisms $\t{\eta}_0,\t{\eta}_1:\C\to\C$ and a neighborhood $U_0$ of the union of all attracting or parabolic cycles of $f_n$ (including the infinity), such that  $$\t{\eta}_0\circ f_n=f_{n+1}\circ \t{\eta}_1\textup{ on }\C,$$ $\t{\eta}_0$ is quasiconformal on a neighborhood of $\overline{U_0}$, $\t{\eta}_0$ is conformal near $\infty$ with $\t{\eta}_0(z)/z\to 1$ as $z\to\infty$ and $\t{\eta}_0,\t{\eta}_1$ are isotopic relative to $P_{f_n}\cup \overline{U_0}$. We can modify $\t{\eta}_0$ to a global quasiconformal map  without changing it on $\overline{U}_0$ such that the modified map $\eta_0$ is isotopic to $\t{\eta}_0$ rel $P_{f_n}\cup \overline{U_0}$. Denote the maximal dilatation of $\eta_0$ by $K$. Then the lift $\eta_1$ of $\eta_0$ by $f_n$ and $f_{n+1}$ is $K$-quasiconformal.
		 
		By the Homotopic Lifting Lemma, we obtain a sequence of $K$-quasiconformal map $\{\eta_k\}_{k\geq0}$ such that $\eta_k\circ f_n=f_{n+1}\circ \eta_{k+1}$ and $\eta_k$ is isotopic to $\eta_{k+1}$ rel $f_n^{-k}(U_0\cup P_{f_n})$. Then $\{\eta_k\}_{k\geq0}$ forms a normal family. By taking a subsequence, we can assume that $\eta_k$ locally uniformly converges to a quasiconformal map $\eta: \mathbb{C} \to \mathbb{C}$. Then $\eta \circ f_n = f_{n+1} \circ \eta$ on $\C$ and $\eta:U_{f_n}(\infty)\to U_{f_{n+1}}(\infty)$ is conformal with $\eta(z)/z\to 1$ as $z\to\infty$. This completes the proof of the former part of Statement (5).
	\vspace{5pt}

{\bf 5. Semiconjugacy}\vspace{5pt}

We continue to prove the latter part of Statement (5). After Step 4, we obtain a map $\t{f}\in\CCCC_d$ and a sequence $\{\eta_n\}$ of quasiconformal map on $\C$ such that for every $n\geq1$, $\eta_n\circ \t{f}=f_n\circ \eta_n$, and 
$\eta_n:U_{\t{f}}(\infty)\to U_{f_{n}}(\infty)$ is conformal with $\eta_n(z)/z\to 1$ as $z\to\infty$.
 As a result,
 \begin{equation}\label{eq:3333}
 \text{ $\eta_n\big(R_{\t{f}}(\theta)\big)=R_{f_{n}}(\theta)$, for any $\theta\in\R/\Z$ and $n\geq1$.}
 \end{equation}

For any polynomial $g$, denote ${\rm Prep}(g)$ the set of preperiodic points of $g$ in the Julia set.\vspace{1pt}

By taking a subsequence, we may assume that $\{\eta_n(\tilde{z})\}$ converges for every $\tilde{z}\in \operatorname{Prep}(\tilde{f})$. This defines a map $\wp:\operatorname{Prep}(\tilde{f})\to\operatorname{Prep}(f)$ by $\wp(\tilde{z})=\lim_{n\to\infty}\eta_n(\tilde{z})$, which satisfies 
\begin{equation}\label{eq:4444}
f\circ\wp(\t{z})=\wp\circ\tilde{f}(\t{z}),\ \forall\, \t{z}\in \operatorname{Prep}(\tilde{f}).
\end{equation}

{\bf Claim 7}. \emph{For any  point $z\in {\rm Prep}(f)$, we have that $\t{z}=\wp^{-1}(z)$ is a singleton unless $f^k(z)=c^*$ for some $k\geq0$, in which case $\wp^{-1}(z)$ consists of $\deg(f^{k+1}, z)$ distinct points.}
\begin{proof}[Proof of Claim 7]
By Statement (3) and Lemma~\ref{lem:stable}, $\wp$ induces a bijection between the repelling (resp.\ parabolic) periodic points of \(\tilde f\) and those of \(f\). If $z\in {\rm Prep}(f)$ is never iterated to $c^*$, then by lifting and Statements (1)--(2), we obtain $\wp^{-1}(z)$ is a singleton.

Now assume \(f^k(z)=c^*\) for some \(k\ge 0\).  
If \(k=0\), then \(z=c^*\) and \(\tilde v^*=\wp^{-1}(v^*)\) is a singleton, where \(v^*=f(c^*)\).  
Since \(c_n^*\in F_{f_n}\) by Statement (1), the point \(z\) is perturbed into \({\rm deg}(f,z)\) distinct preperiodic points of \(f_n\), all mapped to \(\tilde v^*\) by \eqref{eq:4444}.  
By Statement (2) and induction, it follows that \(\wp^{-1}(z)\) consists of \(\deg(f^{k+1}, z)\) distinct points.
\end{proof}

{\bf Claim 8}. \emph{For any point \(\tilde z\in \mathrm{Prep}(\tilde f)\), if \(\pi_{\tilde f}(\theta)=\tilde z\), then \(\pi_f(\theta)=\wp(\tilde z)\). Conversely, if \(z\in\mathrm{Prep}(f)\) satisfies \(\wp^{-1}(z)=\{\tilde z_1,\dots,\tilde z_m\}\), then  
\[
\pi_f^{-1}(z)=\bigsqcup_{j=1}^{m}\pi_{\tilde f}^{-1}(\tilde z_j).
\]}

\begin{proof}[Proof of Claim 8]
From the proof of Claim 7, we have \(\wp\) maps pre‑repelling (resp.\ pre‑parabolic) points of \(\tilde f\) to pre‑repelling (resp.\ pre‑parabolic) points of \(f\).\vspace{3pt}

\noindent\emph{First Conclusion.}  
If \(\tilde z\) is pre‑repelling, the result follows from Lemma~\ref{lem:stable}\,(1) and fact~\eqref{eq:3333}.  
If \(\tilde z\) is pre‑parabolic and periodic, Statement (3) implies that \(\tilde z\) and \(\wp(\tilde z)\) have the same period, multiplier and multiplicity; the claim then follows from Lemma~\ref{lem:stable}\,(2) and \eqref{eq:3333}. The strictly preperiodic case can be obtained by lifting.\vspace{3pt}

\noindent\emph{Second Conclusion.}  
Assume first that \(z\) is periodic. By Claim 7, \(\wp^{-1}(z)=\{\tilde z\}\) is a singleton. The first part gives \(\pi_{\tilde f}^{-1}(\tilde z)\subseteq\pi_f^{-1}(z)\), while Lemma~\ref{lem:stable} together with \eqref{eq:3333} yields the reverse inclusion. Hence the equality holds.

Now suppose \(z\) is strictly preperiodic.  
If \(z\) is never iterated to \(c^*\), then \(\wp^{-1}(z)=\{\tilde z\}\) by Claim 7. Choose \(j\ge1\) such that \(w=f^j(z)\) is periodic, and let \(\tilde w=\tilde f^j(\tilde z)=\wp^{-1}(w)\). The first part gives \(\pi_{\tilde f}^{-1}(\tilde z)\subseteq\pi_f^{-1}(z)\). Since \(\deg(f^j,z)=\deg(\tilde f^j,\tilde z)\) by Statements (1)--(2) and \(\pi_f^{-1}(w)=\pi_{\tilde f}^{-1}(\tilde w)\), the two sets \(\pi_{\tilde f}^{-1}(\tilde z)\) and \(\pi_f^{-1}(z)\) have the same cardinality, therefore  coincide.

If \(f^k(z)=c^*\) for some \(k\ge0\), set \(v^*=f^{k+1}(z)=f(c^*)\). The point \(v^*\) is not iterated to \(c^*\), so by the previous case \(\pi_f^{-1}(v^*)=\pi_{\tilde f}^{-1}(\tilde v^*)\) where \(\tilde v^*=\wp^{-1}(v^*)\). Write \(\wp^{-1}(z)=\{\tilde z_1,\dots,\tilde z_m\}\) with \(m=\deg(f^{k+1},z)\). The first part implies  
\[
\bigsqcup_{j=1}^{m}\pi_{\tilde f}^{-1}(\tilde z_j)\subseteq\pi_f^{-1}(z).
\]  
For each \(j\), \(\#\pi_{\tilde f}^{-1}(\tilde z_j)=\#\pi_{\tilde f}^{-1}(\tilde v^*)=\#\pi_f^{-1}(v^*)\). Hence the two sets have the same cardinality $m\cdot\#\pi_f^{-1}(v^*)$ and therefore coincide. This completes the proof of Claim 8.
\end{proof}

{\bf Claim 9}. \emph{For any $\theta,\eta\in\R\setminus \Q$, $\pi_{f}(\theta)=\pi_f(\eta)$ if and only if $\pi_{\t{f}}(\theta)=\pi_{\t{f}}(\eta)$.}
\begin{proof}[Proof of Claim 9]
By Proposition~\ref{pro:iterate}, every non‑preperiodic point of \(f\) or \(\tilde f\) receives at most two external rays.
We first show that \(\pi_{f}(\theta)=\pi_{ f}(\eta)\) implies \(\pi_{\t{f}}(\theta)=\pi_{\t{f}}(\eta)\).

Assume \(\pi_{f}(\theta)=\pi_{f}(\eta)= z\) is non‑preperiodic. Then \(z\) cannot lie in any maximal Fatou chain of $f$: otherwise it would eventually be mapped  to \(\HHH_{f}\) defined in \eqref{eq:5555}, and hence be preperiodic by Claim 3, a contradiction. Thus, by Lemma~\ref{lem:non-cluster} and Theorem~\ref{thm:CT}, there exists a sequence \(\{x_k\}\) of preperiodic points of \(f\) converging to \(z\), such that each \(x_k\) is pre-repelling, non-precritical, and receives two external rays with arguments \(\theta_k\) and \(\eta_k\). Since \(\pi_{f}^{-1}(z)=\{\theta,\eta\}\), we may assume \(\theta_k\to\theta\) and \(\eta_k\to\eta\) as \(k\to\infty\).  Claim 8 then implies \(\pi_{\t{f}}(\theta_k)=\pi_{\t{f}}(\eta_k)\) for every \(k\). The local connectivity of \(J_{\t{f}}\) now yields \(\pi_{\t{f}}(\theta)=\pi_{\t{f}}(\eta)\).

The converse direction can be proved analogously, and we omit the details.
\end{proof}

Now, we define a map $\phi: J_{\t{f}} \to J_f$ by $\phi(\t{z}) = \pi_f(\theta)$, where $\t{z} = \pi_{\t{f}}(\theta) \in J_{\t{f}}$. 
Since $J_{\t{f}}$ and $J_f$ are locally connected, Claims 8 and 9 imply that $\phi$ is well-defined and continuous. 
Note that $\phi$ coincides with $\wp$ on $\operatorname{Prep}(\t{f})$, a dense subset of $J_{\t{f}}$. 
By the continuity of $\phi$ and equation \eqref{eq:4444}, we obtain that $\phi$ is surjective and $f \circ \phi = \phi \circ \t{f}$ on $J_{\t{f}}$. 
The fiber properties of $\phi$ follow from Claims~7--9. 
This completes the proof of Theorem~\ref{thm:perturbation}.

	\section{Lamination and geodesic lamination}
	
	Let $f$ be a polynomial in $\CCCC_d$. Recall that
	the rational lamination $\lambda_\mathbb{Q}(f)$ of $f$ consists of the pairs $(\theta_1,\theta_2)\in(\Q/\mathbb{Z})^2$ for which the external rays $R_f(\theta_1)$ and $R_f(\theta_2)$ land at the same point.  It is clear that $\lambda_\Q(f)$ is an equivalence relation on $\Q/\Z$.
	
	\begin{lemma}[\cite{K},\,Lemma 3.9]\label{lem:closed}
		The equivalence relation $\lambda_\Q(f)$ is a closed set in $(\Q/\mathbb{Z})^2$.
	\end{lemma}

	Suppose further that  $J_f$ is locally connected.
	The \emph{(real) lamination} $\lambda(f)\subseteq(\mathbb R / \mathbb Z)^2$ of $f$  consists of   $(\theta_1,\theta_2)\in (\mathbb R / \mathbb Z)^2$ for which the external rays $R_f(\theta_1)$ and $R_f(\theta_2)$ land at the same point. Then $\lambda_\Q(f)=\lambda(f)\cap (\Q/\mathbb{Z})^2$. On the other hand, we have

	
	\begin{lemma} [\cite{K},\,Lemma 4.17]
		\label{lem:rational}
		If each critical point of $f$ in the Julia set is preperiodic, then $\lambda(f)$ is
		the smallest closed equivalence relation on $\mathbb R / \mathbb Z$
		that contains $\lambda_\Q(f)$.
	\end{lemma}
	
	The following result is well-known.
	
	\begin{proposition}\label{pro:same1}
		Every subhyperbolic map in $\CCCC_d$ is contained in a relative hyperbolic component, which contains a unique postcritically finite map. Moreover, two subhyperbolic maps in $\CCCC_d$ lie in the same relative hyperbolic component if and only if they have the same rational lamination.
	\end{proposition}
	\begin{proof}[Sketch of the proof]
		Let $f \in \CCCC_d$ be subhyperbolic. If $f$ has no bounded Fatou domains, then it is postcritically finite and then $\HHHH_f = \{f\}$. Otherwise, a standard quasiconformal surgery \cite{DH2} provides a path $\{f_t\}_{0 \leq t \leq 1} \subseteq \CCCC_d$ from $f = f_1$ to a postcritically finite map $f_0$, with each $f_t$ quasiconformally conjugate to $f$ near the Julia set. Hence $f \in \HHHH_{f_0}$.
		
		Let $\HHHH_f$ be the relative hyperbolic component containing $f$. There exists a holomorphic motion $\iota: \overline{U_f(\infty)} \times \HHHH_f \to \mathbb{C}$ such that for all $g \in \HHHH_f$,
		\[
		\iota(\overline{U_f(\infty)}, g) = \overline{U_g(\infty)}, \quad
		\iota(\overline{R_f(\theta)}, g) = \overline{R_g(\theta)} \quad (\theta \in \mathbb{R}/\mathbb{Z}).
		\]
		So all maps in $\HHHH_f$ share the same lamination.
		
		Now take another subhyperbolic $g \in \CCCC_d$ with $\lambda(g) = \lambda(f)$. By the argument above, there exist postcritically finite polynomials $f_0, g_0$ such that $f \in \HHHH_{f_0}$, $g \in \HHHH_{g_0}$, and $\lambda(f_0) = \lambda(g_0)$. By \cite[Theorem 1.1]{Poi1}, we have $f_0 = g_0$.\end{proof}
	
	There is a natural geometric interpretation for laminations, originally  due to Thurston \cite{Th}.
	The map $\sigma_d:\R/\Z\to \R/\Z$ is defined by $\sigma_d(\theta)=d\theta~{\rm mod}~\Z$. We identity $\R/\Z$ with the unit circle $\partial\D$ via the map $\theta\mapsto e^{2\pi i\theta}$.
	
	Let $f\in\CCCC_d$ be a geometrically finite polynomial. The \emph{geodesic lamination}  $L(f)$ of $f$ is defined as the union of $\partial\D$ and all geodesics that form the boundaries of the hyperbolic convex hulls (within $\overline{\mathbb{D}}$) of the  $\lambda(f)$-equivalence classes. Each such geodesic (an open arc) is called a \emph{leaf} of $L(f)$. A leaf with endpoints $a, b\in\partial\mathbb{D}$ is denoted by $\overline{ab}$. Since $\lambda(f)$ is closed, $L(f)$ is a closed subset of $\overline{\D}$.   \vspace{1pt}
	
	A \emph{gap} $\Omega$ of $L(f)$ is a connected component of $\D\setminus L(f)$. The  leaves of $L(f)$ contained in $\partial \Omega$ are called the \emph{edges} of $\Omega$, and the points in $\overline{\Omega}\cap\partial\D$ are its \emph{vertices}. We call $\Omega$ a \emph{finite} or \emph{infinite} gap according to wheter it has finitely or infinitely many edges. Note that if $\Omega$ is an infinite gap, then $\partial\Omega\cap\partial\D$ is a Cantor set; moreover, $\partial\Omega$ is always a Jordan curve.
	
	The projection $\pi_f:\partial \D\to J_f$ extends naturally to $L(f)$: for any leaf $\overline{ab}$ of $L(f)$, we define $\pi_f(\overline{ab})=\pi_f(a)$. Consequently, for any gap $\Omega$ of $L(f)$:
	\begin{enumerate}
		\item[$\bullet$] If $\Omega$ is finite, then $\pi_f(\partial\Omega)$ is a branched point of $J_f$.
		\item[$\bullet$] If $\Omega$ is infinite, then $\pi_f(\partial\Omega)=\partial U$ for some Fatou domain $U$ of $f$.
	\end{enumerate}
	
	The covering $\sigma_d:\partial \D\to \partial\D$ also extends to $\sigma_d:L(f)\to L(f)$. For any leaf $\ell=\overline{ab}$ of $L(f)$: if $\sigma_d(a)=\sigma_d(b)$, we define $\sigma_d(\ell)=\sigma_d(a)$; otherwise, we define $\sigma_d(\ell)=\overline{\sigma_d(a)\sigma_d(b)}$. The extension satisfies $\pi_f\circ \sigma_d=f\circ\pi_f$ on $L(f)$. Furthermore,  if $\sigma_d(\partial\Omega)=\partial \Omega'$ for two gaps $\Omega,\Omega'$ of $L(f)$, we define $\sigma_d(\Omega)=\Omega'$.
	
	Let $L_\Q(f)$ denote the union of $\partial\D$ and all leaves of $L(f)$ with rational angles. Then the following result follows directly from Proposition \ref{pro:iterate} and Lemma \ref{lem:rational}.
	
	\begin{corollary}\label{coro:same}
		Let $f,g\in\CCCC_d$ be  geometrically finite polynomials. Then  $L_\Q(f)$ contains the edges of all gaps of $L(f)$, and $\lambda(f)=\lambda(g)$ provided that $L_\Q(f)=L_\Q(g)$.
	\end{corollary}
	
	Moreover, Lemma \ref{lem:non-cluster} can be formulated in the context of laminations. A disk $S \subseteq \mathbb{D}$ is called a \emph{strip} if its boundary  consists of two hyperbolic geodesics in $\D$, referred to as the \emph{edges} of $S$, together with two disjoint arcs on $\partial \mathbb{D}$. 
	\begin{lemma}\label{lem:lamination-version}
		Suppose $f\in\CCCC_d$ is a geometrically finite polynomial. Let $\ell=\overline{\alpha\beta}$ and $\ell'=\overline{\alpha'\beta'}$ be two leaves of $L(f)$ such that $\pi_f(\alpha)$ and $\pi_f(\alpha')$ do not lie in the same maximal Fatou chain.   Let $S_0\subseteq \D$ be the strip bounded by $\ell$ and $\ell'$. Then there exist strips $S',S,S_1,S_2\subseteq \D$, and positive integers $l,m_1,m_2$, satisfying
		\begin{enumerate}
			\item $S'\subseteq S_0$ and it separates the two edges of $S_0$;\vspace{1pt}
			\item both edges of $S'$ are leaves of $L_\Q(f)$ avoiding all gap boundaries in $L(f)$, and $\sigma_d^l:\partial S'\to\partial S$ is a homeomorphism;\vspace{1pt}
			\item $\overline{S_1}\cap \overline{S_2}=\emptyset$, and for each $i\in\{1,2\}$,
			\begin{itemize}
				\item $S_i\subseteq S$ and separates the two edges of $S$;
				\item   $\sigma_d^{m_i}:\partial S_i\to\partial S$ is a homeomorphism.
			\end{itemize}
		\end{enumerate}
		
		As a consequence, if  $\pi_f(\alpha)=z$ either avoids all maximal Fatou chains, or is an endpoint of a maximal Fatou chain with $\pi_f^{-1}(z)=\{\alpha,\beta\}$, then there exists a sequence $\{\ell_n\} $ of leaves in $L_\Q(f)$, each avoiding gap boundaries, such that $\ell_n \to \ell$ and the periods of $\ell_n$ tend to infinity as $n \to \infty$. Moreover,  $\{\ell_n\}$ can be chosen in $S_0$ if  $S_0$ avoids the gap of $L(f)$ whose boundary contains $\ell$.
		\end{lemma}
	\begin{proof}
		By Theorem \ref{thm:CT}, we may assume that $f$ is postcritically finite.
		
		For an open arc $\gamma \subseteq K_f$ such that its endpoints  $z_1, z_2$ are cut points of $J_f$, we associate to $\gamma$ a strip $S_\gamma$ as follows. For each $i = 1, 2$, choose angles $\alpha_i,\beta_i$ such that $\pi_f(\alpha_i)=\pi_f(\beta_i) = z_i$ and such that the component of $\mathbb{C} \setminus (R_f(\alpha_i) \cup \{z_i\} \cup R_f(\beta_i))$ containing $\gamma$ is disjoint from all external rays landing at $z_i$. Then $S_\gamma$ is the strip bounded by the leaves $\overline{\alpha_1\beta_1}$ and $\overline{\alpha_2\beta_2}$ of $L(f)$.
		
		Now, let $I_0\subseteq K_f$ be a regulated arc with  endpoints  $\pi_f(\alpha)$ and $\pi_f(\alpha')$. Then $I_0$ does not entirely lie in any maximal Fatou chain.
		We may apply Lemma \ref{lem:non-cluster} to obtain arcs $I'\subseteq I_0, I, I_1,I_2$, and positive integers $l,m_1,m_2$, satisfying the corresponding properties. It follows that the strips $S'=S_{I'},S=S_{I}, S_1=S_{I_1}$ and $S_2=S_{I_2}$, together with the integers $l,m_1$ and $m_2$ satisfy the conclusions of the lemma.
		
If $\pi_f(\alpha)=z$ satisfies the condition of the lemma, then there exists a component $W$ of $\C\setminus \overline{R_f(\alpha)\cup R_f(\beta)}$, satisfying 
\begin{itemize}
\item no external rays in $W$ land at $z$; and
\item there exists  a  regulated arc $\g:[0,1]\to K_f\cap \overline{W}$ with $\g(0)=z$, such that  $\g(0,t)$ does not lie in any maximal Fatou chain for every $t>0$.
\end{itemize}
For each  $n\geq1$, set $I(n):=\g[0,1/n]$.   Then we can apply the result proved above to $S_0=S_{I(n)}$, and then obtain a leaf $\ell_n\subseteq S_{I(n)}$ in $L_\Q(f)$ with period larger than $n$. Thus $\ell_n\to \ell$ as $n\to\infty$.

To verify that these \(\ell_n\) avoid gap boundaries of \(L(f)\), it suffices to show that periods of leaves lying on gap boundaries are uniformly bounded. Let \(\ell_1\) be a leaf of \(L_\mathbb Q(f)\) contained in a gap boundary, and set \(w=\pi_f(\ell_1)\).  
If the gap is finite, then \(w\) is a branch point of \(J_f\). By Proposition~\ref{pro:iterate} it is eventually mapped into the vertex set of \(T_g\).  
If the gap is infinite, then \(w\) is eventually mapped to the intersection of the boundaries of periodic Fatou domains with \(T_g\), which is also a finite set.  
In both cases the period of \(\ell_1\) is uniformly bounded, which completes the proof.
		\end{proof}

A polynomial $g\in\CCCC_d$ is called \emph{primitive} if either $J_g=K_g$ or the bounded Fatou domains of $g$ have pairwise disjoint closures, and \emph{non-primitve} otherwise. The maximal Fatou chains of primitive polynomials are exactly the closures of the bounded Fatou domains.

\begin{lemma}\label{lem:include}
Let $f_0$ be a primitive postcritically finite polynomial and $f\in\CCCC_d$. Suppose that $(\alpha,\beta)\in\lambda_\Q(f)$ for each leaf $\overline{\alpha\beta}$ of $L_\Q(f_0)$ avoiding gap boundaries. Then $\lambda_\Q(f_0)\subseteq \lambda_\Q(f)$.
\end{lemma}
\begin{proof}
Let $\ell=\overline{\alpha\beta}$ be a leaf on the boundary of a gap of $L(f_0)$. Then $\ell\subseteq L_\Q(f_0)$ by Corollary \ref{coro:same}. It suffices to show that $(\alpha,\beta)\in\lambda_\Q(f)$.

If $\ell$ is not a common edge of two gaps of $L(f_0)$, then the latter part of Lemma \ref{lem:lamination-version} gives a sequences of leaves $\{\ell_n=\overline{\alpha_n\beta_n}\}$ in $L_\Q(f_0)$ avoiding gap boundaries and converging to $\ell$. Since $(\alpha_n,\beta_n)\in\lambda_\Q(f)$, it follows from Lemma \ref{lem:closed}  that $(\alpha,\beta)\in \lambda_\Q(f)$.

Suppose that $\ell$ is the common edge of two gaps of $L(g_0)$. Since $f_0$ is primitive, one of the two gaps is infinite, and the other one is finite, denoted by $G$, such that all leaves $\ell_j=\overline{\theta_j\eta_j}, j=1,\ldots,m$ of $G$ other than $\ell$ lie in the boundary of exactly one gap $G$. The previous argument shows that $(\theta_j,\eta_j)\in\lambda_\Q(f)$ for each $j$. Thus $(\alpha,\beta)\in \lambda_\Q(f)$.
\end{proof}

	Fix a primitive postcritically finite polynomial $f_0$. Suppose that $f$ is a polynomial in $\CCCC_d$ such that $\lambda_\Q(f_0)\subseteq \lambda_\Q(f)$. Let $\Omega$ be a gap of $L(f_0)$.
	
	If $\Omega$ is a finite gap, then all its vertices are rational by Corollary \ref{coro:same}, and hence $\pi_f(\partial\Omega)$ is a singleton due to the inclusion  $\lambda_\Q(f_0)\subseteq \lambda_\Q(f)$. In this case, we set $\pi_f(\overline{\Omega})=\pi_f(\partial\Omega)$.
	
	Suppose now that $\Omega$ is an infinite gap. For any edge $\overline{ab}$ of $\Omega$, denote by $]a,b[$ the component of $\partial\D\setminus\{a,b\}$ with $]a,b[\,\cap \partial\Omega=\emptyset$. By Corollary \ref{coro:same}, we have $\overline{ab}\in L_\Q(f_0)$.
	Define
	\begin{equation}\label{eq:23}
		\pi_f(\overline{\Omega})=K_{f,\Omega}:= K_f\setminus \bigsqcup_{\ell\subseteq\partial \Omega}W_f(\ell),\end{equation}
	where $\ell=\overline{ab}$ ranges over all edges of $\Omega$, and $W_f(\ell)$ denotes the component of $\C\setminus\overline{R_f(a)}\cup\overline{R_f(b)}$ containing $R_f(\theta), \theta\in]a,b[$.
	Note that $\pi_{f_0}(\overline{\Omega})$ is the closure of a Fatou domain $U$ of $f_0$, and $K_{f,\Omega}$ coincides with $K_{f,U}$ defined in \eqref{eq:11}.

	\begin{proposition}\label{pro:ren}
	Let $f_0$ be a primitive postcritically finite polynomial such that $\lambda_\Q(f_0)\subseteq\lambda_\Q(f)$.	Then for distinct infinite gaps $\Omega$ and $\Omega'$ of $L(f_0)$, $K_{f,\Omega}$ and $K_{f,\Omega'}$ are disjoint full continua. Moreover,
		if $\sigma_d(\Omega')=\Omega$, then $K_{f,\Omega'}$ is a component of $f^{-1}(K_{f,\Omega})$. 	\end{proposition}
	\begin{proof}
		First, note that $K_{f,\Omega}$ is the intersection of a nested sequence of closed disks, hence a full continuum. Since $\overline{\Omega}\cap\overline{\Omega'}=\emptyset$, there exist leaves $\ell\subseteq\partial\Omega$ and $\ell'\subseteq\partial\Omega'$ such that the strip $S_0$ bounded by $\ell$ and $\ell'$ separates $\Omega$ and $\Omega'$. Applying Lemma \ref{lem:lamination-version} to $f_0$ and $S_0$, we obtain a leaf $\overline{\theta\eta}\subseteq S_0$ in $L_\mathbb{Q}(f_0)$ with a sufficiently large period. Since $(\theta,\eta)\in \lambda_\Q(f)$, the union $\overline{R_f(\theta)} \cup \overline{R_f(\eta)}$ forms an arc separating $K_{f,\Omega}$ from $K_{f,\Omega'}$. Consequently, $K_{f,\Omega} \cap K_{f,\Omega'} = \emptyset$.
		
		Let  $\Omega_1 = \Omega', \Omega_2, \ldots, \Omega_m$ be all gaps of $L(f_0)$ satisfying $\sigma_d(\Omega_i) = \Omega$. Then $f(K_{f,\Omega_i}) = K_{f,\Omega}$ for each $i$. Since the sets $K_{f,\Omega_1}, \ldots, K_{f,\Omega_m}$ are pairwise disjoint, it follows that $K_{f,\Omega'}=K_{f,\Omega_1}$ is a component of $f^{-1}(K_{f,\Omega})$.
%
	\end{proof}
	
	We now consider the convergence of laminations.	
	Throughout the remainder of this section, we fix a primitive postcritically finite polynomial $f_0$ and set $\lambda_{\Q}=\lambda_{\Q}(f_0)$.
	
	\begin{proposition}\label{pro:lami-convergence}
		Let $\{f_n\}_{n\geq1}\subseteq \CCCC_d$ be a sequence of maps converging to $g$.
		\begin{enumerate}
			\item If $\lambda_\Q\subseteq \lambda_\Q(f_n)$ for all $n$, then $\lambda_\Q\subseteq\lambda_\Q(g)$.
			\item If $\lambda_\Q\subseteq \lambda_\Q(g)$ and all $f_n$ lie in the same relative hyperbolic component, then $\lambda_\Q\subseteq \lambda_\Q(f_1)$.
		\end{enumerate}
	\end{proposition}
	\begin{proof}
		There exists a positive number $N_g$ such that, for any $\theta \in \mathbb{Q}/\mathbb{Z}$ with $\sigma_d$-period greater than $N_g$, the landing point of $R_g(\theta)$ is pre-repelling and its orbit avoids the critical points of $g$.
		
		(1) By Lemma \ref{lem:include}, it suffices to show that $(\alpha,\beta) \in \lambda_\mathbb{Q}(g)$ for any leaf $\overline{\alpha\beta}$ of $L_\mathbb{Q}(f_0)$ avoiding gap boundaries. In this case, $\pi_{f_0}(\alpha)$ avoids maximal Fatou chains of $f_0$.
So by the latter part of Lemma~\ref{lem:lamination-version}, there exists a sequence $\{\overline{\alpha_k\beta_k}\}$ of leaves in $L_\mathbb{Q}(f_0)$ converging to $\overline{\alpha\beta}$ such that the $\sigma_d$-period of $\overline{\alpha_k\beta_k}$ exceeds $N_g$ for each $k\geq1$.
		Since $(\alpha_k,\beta_k)\in \lambda_\mathbb{Q}(f_n)$ for all $n$, Lemma~\ref{lem:stable}\,(1) implies $(\alpha_k,\beta_k) \in \lambda_\mathbb{Q}(g)$. As $\lambda_\Q(g)$ is closed in $(\Q/\Z)^2$ by Lemma~\ref{lem:closed}, we conclude that $(\alpha,\beta) \in \lambda_\mathbb{Q}(g)$.\vspace{3pt}
		
		(2) By Lemma \ref{lem:include}, it suffices to show that $(\alpha,\beta) \in \lambda_\mathbb{Q}(f_{1})$ for any leaf $\overline{\alpha\beta}$ of $L_\mathbb{Q}(f_0)$ with $\sigma_d$-period greater than $N_g$.
		In this case, $R_g(\alpha)$ and $R_g(\beta)$ land at a common pre-repelling point whose orbit avoids critical points. By Lemma~\ref{lem:stable}\,(1), we have $(\alpha,\beta) \in \lambda_\mathbb{Q}(f_n)$ for all sufficiently large $n$. Since all $f_n$ belong to the same relative hyperbolic component, Proposition~\ref{pro:same1} implies $(\alpha,\beta) \in \lambda_\mathbb{Q}(f_{1})$, completing the proof of Statement (2).
	\end{proof}
	
	By a \emph{finite chain of (relative) hyperbolic components}, we mean the union of closures of finitely many (relative) hyperbolic components $\HHHH_1,\ldots,\HHHH_m$ such that $\overline{\HHHH_i}\cap \overline{\HHHH_{i+1}}\not=\emptyset$ for $i=1,\ldots,m-1$.
	
	\begin{proposition}\label{pro:transitive1}
		Suppose that $f$ and $g$ are postcritically finite polynomials such that $\HHHH_f$ and $\HHHH_g$ are contained in a finite chain of relative hyperbolic components. Then $\lambda_\Q\subseteq \lambda_\Q(f)$ if and only if $\lambda_\Q\subseteq \lambda_\Q(g)$.
	\end{proposition}
	\begin{proof}
		By symmetry, it suffices to prove that $\lambda_\mathbb{Q} \subseteq \lambda_\mathbb{Q}(f)$ implies $\lambda_\mathbb{Q} \subseteq \lambda_\mathbb{Q}(g)$.
		By induction, we may assume $\overline{\HHHH_f} \cap \overline{\HHHH_g} \neq \emptyset$.
		Let $\t{f}\in \overline{\HHHH_f}\cap\overline{\HHHH_g}$. The application of Proposition \ref{pro:lami-convergence}\,(1) to $\HHHH_f$ and $\t{f}$ gives $\lambda_\Q\subseteq \lambda_\Q(\t{f})$, while applying Proposition \ref{pro:lami-convergence}\,(2) to $\t{f}$ and $\HHHH_g$ yields $\lambda_\Q\subseteq \lambda_\Q(g)$.
	\end{proof}
	
	\section{The extended molecules}\label{sec:B}
	

	For any bounded hyperbolic component $\HHHH$, the \emph{extended molecule} $\MMMM_+=\MMMM_+(\HHHH)$  is defined as the closure of the union of all relative hyperbolic components which are connected to $\HHHH$ by finite chains of relative hyperbolic components.

	In this section, we shall characterize the geometrically finite maps in extended molecules.
	\begin{theorem}\label{thm:gf}
		Every extended molecule $\MMMM_+$ contains a unique primitive hyperbolic component $\HHHH$, and it holds that $\lambda_\Q(f_0)\subseteq \lambda_\Q(f)$ for every $f\in \MMMM_+$, where $f_0$ is the unique postcritically finite map in $\HHHH$. Furthermore, a geometrically finite polynomial $f$ lies in $\MMMM_+$ if and only if
		\begin{enumerate}
			\item $\lambda_\Q(f_0)\subseteq \lambda_\Q(f)$; and
			\item  for any infinite gap $\Omega$ of $L(f_0)$,  the set $K_{f,\Omega}$   in \eqref{eq:23} is a maximal Fatou chain of $f$.
		\end{enumerate}

		
	\end{theorem}
	
	If $\HHHH$ is the main hyperbolic component,  then this theorem is exactly Theorem LP due to Propositions \ref{pro:iterate} and \ref{pro:cluster}. Our proof of Theorem \ref{thm:gf} applies the hyperbolic-parabolic deformation theory developed in \cite{CT2} and Theorem \ref{thm:perturbation}, differing from the technique in \cite{LP}.
	
	\subsection{Hyperbolic-parabolic deformation}\label{sec:appendix}
	We borrow several key results from the general theory of hyperbolic-parabolic deformation in \cite{CT2} and adapt them to polynomial settings.
	
	\begin{proposition}[Simple pinching \textup{\cite[Theorem 1.5]{CT2}}]\label{pro:pinching}
		Let $f\in\CCCC_d$ be a subhyperbolic polynomial. Let $x_1,\ldots, x_m$ be attracting periodic points of $f$ with disjoint orbits. For each $k\in\{1,\ldots,m\}$, let $z_k$ be a periodic point on the boundary of the immediate basin of $x_k$ such that  the period of  $z_k$ does not exceed that of $x_k$. Then there exist a quasiconformal path $\{f_t\}_{t\geq0}$ with $f=f_0$ and a continuous onto map $\phi:\C\to\C$, such that
		\begin{enumerate}
			\item $f_t$ converges uniformly to a geometrically finite map $g$ as $t\to\infty$;
			
			\item $\phi: J_f\to J_g$ is a topological conjugacy between between $f:J_f\to J_f$ and $g:J_g\to J_g$;

			\item for each $k$, $\phi(x_k) = \phi(z_k)$ is a parabolic point of $g$ with period equal to that of $z_k$.
		\end{enumerate}
	\end{proposition}

	Conversely, given a geometrically finite map $g$, there are various quasiconformal paths (converging) to $g$.
	These pinching paths are uniquely determined by the ``plumbing combinatorics'' of $g$. The trivial one is called the \emph{simple plumping} combinatorics.
	
	\begin{proposition}[Simple plumbing \textup{\cite[Theorem 1.4]{CT2}}]\label{pro:plumbing1}
		Let $g\in\CCCC_d$ be a geometrically finite polynomial. Then there exist a quasiconformal path $\{f_t\}_{t\geq0}$ starting from a subhyperbolic map $f=f_0\in\CCCC_d$ and converging to $g$, and a continuous onto map $\phi:\C\to\C$, such that
		$\phi: J_f\to J_g$ is a topological conjugacy  between $f:J_f\to J_f$ and $g:J_g\to J_g$.
	\end{proposition}
	
	We describe two more particular forms of plumbing combinatorics; see Figure \ref{plumbing}.

\begin{figure}[http]
	\centering
	\begin{tikzpicture}
		\node at (0,0){ \includegraphics[width=14cm]{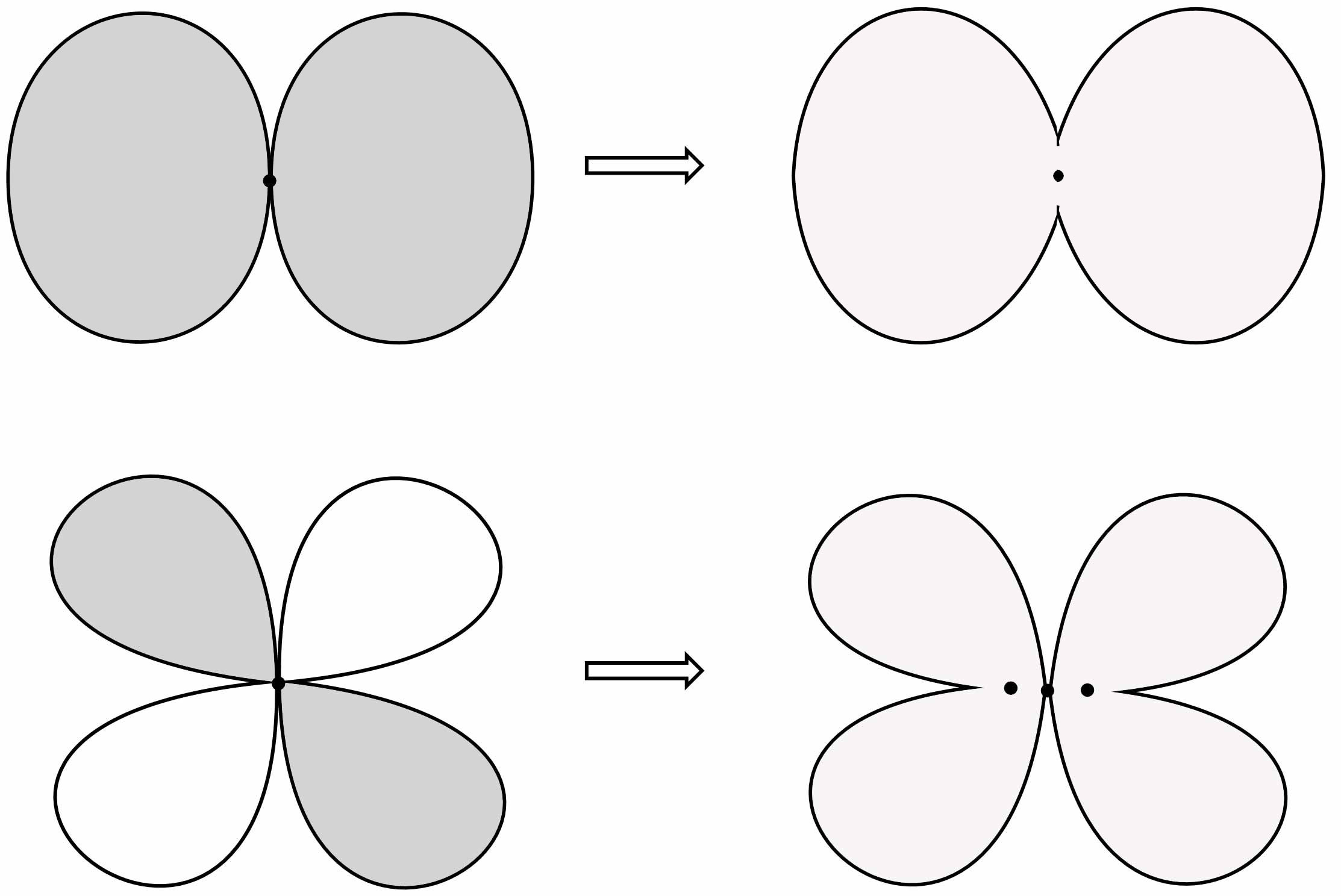}};
		\node at (-0.25, 3.5){Case (i)};
		\node at (-0.25, -1.75){Case (ii)};
	\end{tikzpicture}
	\caption{Illustration of Proposition \ref{pro:plumbing2}.}\label
	{plumbing}
\end{figure}		
	
	\begin{proposition}[Plumbing \textup{\cite[Theorem 1.6]{CT2}}]\label{pro:plumbing2}
		Let $g\in\CCCC_d$ be a geometrically finite polynomial with a unique parabolic  cycle $\mathcal{O}_{g}$ of period $p$. Suppose 
		\begin{itemize}
			\item[(i)] $\mathcal{O}_g$ attracts only one cycle of parabolic domains, of period $q>p$; or
			\item[(ii)] $\mathcal{O}_g$ attracts exactly two  cycles of parabolic domains, of period $q\geq p$.
		\end{itemize}
		Then, there exist a quasiconformal path $\{f_t\}_{t\geq0}$ starting from a subhyperbolic map $f=f_0\in \CCCC_d$ and converging to $g$,  and a continuous onto map $\phi:\C\to\C$, such that
		\begin{enumerate}
			\item $\phi: J_f\to J_g$ is a semi-conjugacy from $f:J_f\to J_f$ to $g:J_g\to J_g$;\vspace{2pt}
			\item $f$ has a unique attracting  cycle $\mathcal{O}_f$  with $\phi(\mathcal{O}_f)=\mathcal{O}_g$, and its period  is $p$ in case (i) and $q$ in case (ii).
		\end{enumerate}
	\end{proposition}

\subsection{Reduction along relative hyperbolic components}

For a geometrically finite map $g\in \CCCC_d$, its \emph{complexity}  is defined as
\[\kappa(g)=\#\{\text{all periodic Fatou domains of $g$}\}+\#\big(C_g\cap J_g\big).\]
It is straightforward to verify the following result.

\begin{proposition}\label{pro:decrease}
	Let $\{f_t\}_{t\geq0}$ be a quasiconformal path starting from  a subhyperbolic map $f=f_0\in\CCCC_d$ and converging to $g$. If $g$ and $\{f_t\}$ satisfy the conditions of Proposition \ref{pro:pinching}, then $\kappa(f)=\kappa(g)$. If $g$ and $\{f_t\}$ satisfy the conditions of Proposition \ref{pro:plumbing2}, then $\kappa(f)<\kappa(g)$.
\end{proposition}

For a subhyperbolic map $f\in\CCCC_d$,
$\HHHH_{f}$ denotes the relative hyperbolic component containing $f$.

\begin{lemma}\label{lem:reduce}
	Let  $f\in\CCCC_d$ be a non-primitive subhyperbolic map. Then there exists a subhyperbolic map $ f_1\in\CCCC_d$ satisfying
	$\ov{\HHHH_{f}}\cap\ov{\HHHH_{f_1}}\not=\emptyset$ and $\kappa(f_1)<\kappa(f)$. Furthermore,
	for any $(\alpha,\beta)\in\lambda_\Q(f)$ such that $\pi_f(\alpha)$ avoid  maximal Fatou chains of $f$,  we have that $(\alpha,\beta)\in\lambda_\Q(f_1)$ and $\pi_{f_1}(\alpha)$  avoids maximal Fatou chains of $f_1$.
	In particular, if $f$ is hyperbolic, then $f_1$ is also hyperbolic.
\end{lemma}
\begin{proof}
	Since $f$ is non-primitive, there are two bounded Fatou domains $U$ and $V$ of $f$ such that their boundaries intersect at some point $y$. Then either $y$ is (pre-)critical, or $f^k(U) \neq f^k(V)$ for all $k \geq 1$. Thus at least one of the following holds:
	\begin{enumerate}
		\item A periodic point $z$ lies on the boundaries of at least two bounded Fatou domains, and all these Fatou domains belong to the same cycle;
		
		\item The boundaries of two bounded Fatou domains $U_1$ and $U_2$ from distinct cycles intersect at a periodic point $z$;
		
		\item A critical point $c$ lies on the boundary of a bounded Fatou domain.
	\end{enumerate}
	We will prove the lemma by considering each of these cases separately.\vspace{3pt}
	
	Case (1). Let $U$ be a Fatou domain of $f$ containing an attracting periodic point $x$ such that $z \in \partial U$. By applying Proposition
	\ref{pro:pinching} to $f$ and $\{x,z\}$, we obtain a geometrically finite map $g\in\partial \HHHH_f$ with a unique parabolic cycle $\OOO_g$, and a continuous  map $\phi$ on $\C$ such that $\phi:J_f\to J_g$ is a homeomorphism and $\phi(x)=\phi(z)\in\OOO_g$.
	The map $\phi$ sends maximal Fatou chains of $f$ to those of $g$. Consequently, $\pi_g(\alpha)$ and $\pi_g(\beta)$ avoid the maximal Fatou chains of $g$, and are eventually repelling. Thus, Lemma \ref{lem:stable}\,(1) implies $(\alpha,\beta)\in\lambda_\Q(g)$.

	Note that $g$ satisfies the hypothesis of Proposition~\ref{pro:plumbing2}\,(i).
	We thus obtain a subhyperbolic map $f_1 \in\CCCC_d$ with a marked attracting cycle $\mathcal{O}_{f_1}$ and a continuous map $\phi_1: \mathbb{C} \to \mathbb{C}$, such that  $\phi_1(J_{f_1})=J_g,\,\phi_1(\OOO_{f_1})=\OOO_g$ and $g\in\partial\HHHH_{f_1}$.

	We conclude from Proposition \ref{pro:decrease} that $\kappa(f_1)<\kappa(g)=\kappa(f)$. Moreover, the map $\phi_1$ lifts the maximal Fatou chains of $g$ to those of $f_1$ and is injective elsewhere. Since $\phi_1(\pi_{f_1}(\alpha))=\pi_g(\alpha)$ and $\phi_1(\pi_{f_1}(\beta))=\pi_g(\beta)$, and since $\pi_g(\alpha)=\pi_g(\beta)$ avoids the maximal Fatou chains of $g$, it follows that $\pi_{f_1}(\alpha)=\pi_{f_1}(\beta)$, and this point is disjoint from the maximal Fatou chains of $f_1$.
	
	If $f$ is hyperbolic, then $g$ has no critical points in its Julia set, which implies that $f_1$ is hyperbolic. This completes the proof of the lemma in Case (1).
	\vspace{2pt}
	
	Case (2). Let $x_1$ and $x_2$ be the attracting points in $U_1$ and $U_2$, respectively. We apply Proposition \ref{pro:pinching} to $f$ and the pairs $\{x_1,x_2\}, \{z_1,z_2\}$ with $z_1=z_2=z$, obtaining a geometrically finite map $g\in\partial\HHHH_f$ with a unique parabolic cycle. The rest argument parallels Case (1), except that $g$ now satisfies the hypothesis of Proposition \ref{pro:plumbing2}\,(ii). We omit the details.
	\vspace{2pt}
	
	Case (3). Without loss of generality, assume that $c$ is a last critical point of $f$. Applying Theorem \ref{thm:perturbation}\,(5) to $(f,c)$ yields a subhyperbolic map $f_1\in\CCCC_d$ with $f\in\partial\HHHH_{f_1}$ and a semiconjugacy $\phi$ from $f_1:J_{f_1}\to J_{f_1}$ to $f:J_f\to J_f$. The properties of $\phi$ implies $\#(C_f\cap J_f)>\#(C_{f_1}\cap J_{f_1})$, hence $\kappa(f)>\kappa(f_1)$. Moreover, letting $z:=\pi_f(\alpha)$, the point $z_1=\phi^{-1}(z)$ is disjoint from the maximal Fatou chains of $f_1$ and satisfies $\pi_f^{-1}(z)=\pi_{f_1}^{-1}(z_1)$. Consequently, $(\alpha,\beta)\in\lambda_\Q(f_1)$.	
%
%
\end{proof}

\begin{proposition}\label{pro:reduction}
	Let $f\in\CCCC_d$ be a subhyperbolic polynomial. Then, there exists a unique primitive postcritically finite polynomial $f_0$ such that $\HHHH_f$ and $\HHHH_{f_0}$ are contained in a finite chain of relative hyperbolic components. Moreover, we have
	\begin{enumerate}
		\item $\lambda_\Q(f_0)\subseteq \lambda_\Q(f)$, and for any infinite gap $\Omega$ of $L(f_0)$ (if exists), the set $K_{f,\Omega}$ defined in \eqref{eq:23}  is a maximal Fatou chain of $f$;
		\item if $(\alpha,\beta)\in\lambda_\Q(f)$ and $\pi_f(\alpha)$ avoids  maximal Fatou chains, then  $(\alpha,\beta)\in\lambda_\Q(f_0)$;
		\item if each critical point of $f$ lies in a maximal Fatou chain, then $f_0$ is hyperbolic;
		\item if $f$ is hyperbolic, then $\HHHH_f$ and $\HHHH_{f_0}$ lie in a finite chain of hyperbolic components.
	\end{enumerate}
\end{proposition}
\begin{proof}
	By successive applying Lemma \ref{lem:reduce}, we obtain a sequence of subhyperbolic maps $f=f_m,\ldots, f_0\in\CCCC_d$ such that $\overline{\HHHH_{f_{i+1}}}\cap \overline{\HHHH_{f_i}}\not=\emptyset$ for $i=m-1,\ldots,0$, the polynomial $f_0$ is primitive and postcritically finite, and that Statements (2) and (4) of the proposition are satisfied.

	By Proposition \ref{pro:transitive1}, we have $\lambda_\Q(f_0) \subseteq \lambda_\Q(f)$. To prove (1), it suffices to verify the property for every infinite periodic gap $\Omega$ of $L(f_0)$. Proposition \ref{pro:ren} implies $K_{f,\Omega}$ is a full stable continuum.
		
	Fix  $g_i\in \overline{\HHHH_{f_i}}\cap\overline{\HHHH_{f_{i+1}}}$ for every $i=0,\ldots,m-1$. Then there exists $N>0$ such that for any $\theta\in\Q/\Z$ with period larger than $N$, the point $\pi_{g_i}(\theta)$ is pre-repelling and its orbit avoids the critical points of $g_i$, for every $i=0,\ldots,m-1$.

	Assume that $K_{f,\Omega}$ is not a maximal Fatou chain. Then it is also not contained in a maximal Fatou chain by Proposition \ref{pro:ren}.  Due to Proposition \ref{pro:cluster}, there exists a periodic cut point $z$ of $K_{f,\Omega}$ with period larger than $N$. Choose angles $\alpha,\beta\in\partial\Omega\cap\partial\D$ such that $R_f(\alpha)$ and $R_f(\beta)$ land at $z$ and $\overline{\alpha\beta}\subseteq\Omega$.  Successively applying Lemma \ref{lem:stable}\,(1) to each triple $\{\HHHH_{f_{i+1}},g_i,\HHHH_{f_i}\}$ ($i=m-1,\ldots,0$) shows that
	$R_{f_0}(\alpha)$ and $R_{f_0}(\beta)$ also land at a common point, contradicting that $\Omega$ is a gap of $L(f_0)$. This completes the proof of Statement (1).
	
	Suppose every critical point of $f$ lies in a maximal Fatou chain. Then by Statement (1) and a degree argument,  each critical point of $f_0$ lies in $K_{f_0,\Omega}$ for some infinite gap $\Omega$ of $L(f_0)$, which is just the closure of a bounded Fatou domain of $f_0$. Since $f_0$ is primitive, it follows that  $f_0$ is hyperbolic. This completes the proof of Statement (3).
	
	The uniqueness of $f_0$ follows from  Propositions \ref{pro:same1} and \ref{pro:transitive1}.
\end{proof}

\subsection{Geometrically finite maps in extended molecules}
This section is devoted to proving Theorem \ref{thm:gf}.
From the definition of extended molecules, we immediately obtain the following.
	
	\begin{proposition}\label{pro:def}
	Given an extended molecule $\MMMM_+=\MMMM_+(\HHHH)$, any bounded hyperbolic component is either disjoint from $\MMMM_+$ or connected to $\HHHH$ by a finite chain of relative hyperbolic components. As a consequence, we have $\MMMM_+=\MMMM_+(\HHHH')$ for any hyperbolic component $\HHHH'\subseteq\MMMM_+$.
	\end{proposition}
	
\begin{proof}[Proof of Theorem \ref{thm:gf}]
	Let $\MMMM_+$ be an extended molecule. From Propositions \ref{pro:def} and \ref{pro:reduction}, it follows that $\MMMM_+$ contains a unique primitive hyperbolic component.
	
	Let $f_0$ be the unique postcritically finite map in this primitive hyperbolic component.
	By definition, for any $f\in\MMMM_+$, there exists a sequence $\{f_n\}_{n\geq1}$ of subhyperbolic maps converging to $f$, each connected to $f_0$ by a finite chain of relative hyperbolic components. By Proposition \ref{pro:transitive1}, we have $\lambda_\Q(f_0)\subseteq \lambda_\Q(f_n)$ for all $n$, and Proposition  \ref{pro:lami-convergence}\,(1) then yields $\lambda_\Q(f_0)\subseteq \lambda_\Q(f)$. \vspace{4pt}
	
	We start to check the  properties for  geometrically finite polynomials lying in $\MMMM_+$. \vspace{5pt}
	
	{\bf Claim 1}. \emph{Let $f\in\CCCC_d$ be a (sub)hyperbolic map with properties (1) and (2) in Theorem \ref{thm:gf}. Then $f$ and $f_0$ lie in  a finite chain of (relative) hyperbolic components.}
	
	\begin{proof}[Proof of Claim 1]
		
		By Proposition \ref{pro:reduction}, there exists a unique primitive hyperbolic postcritically finite polynomial $g_0$, such that $\HHHH_f$ and $\HHHH_{g_0}$ are connected by a finite chain of relative hyperbolic components with $\lambda_\Q(g_0)\subseteq\lambda_\Q(f)$, and such that
		\vspace{5pt}
		
		\qquad \quad \qquad\qquad\emph{the continuum $K_{f,\Omega(g_0)}$ is  a maximal Fatou chain of $f$ \quad\qquad\qquad\qquad$(\ast)$}
		\vspace{5pt}
		
		\noindent for every infinite gap $\Omega(g_0)$ of $L(g_0)$. In particular, if $f$ is hyperbolic, every relative hyperbolic component in this chain is a hyperbolic component.
		
		Note that Proposition \ref{pro:transitive1} implies $\lambda_\Q(f_0)\subseteq \lambda_\Q(g_0)$. Together with property (2) of Theorem \ref{thm:gf} and Statement $(\ast)$, this gives $L(f_0)=L(g_0)$. Hence $f_0=g_0$ by Corollary \ref{coro:same} and Proposition \ref{pro:same1}, proving the claim.
	\end{proof}

	Assume first that $f\in\CCCC_d$ is a geometrically finite map satisfying properties (1) and (2) of Theorem \ref{thm:gf}. By Proposition \ref{pro:plumbing1}, there exists a subhyperbolic polynomial $g$ satisfying (1) and (2) such that  $f\in\overline{\HHHH_g}$. The earlier claim then gives $f\in\MMMM_+$.
	
	Now suppose that $f$ is a geometrically finite polynomial in $\MMMM_+$. We showed that $\lambda_\Q(f_0)\subseteq\lambda_\Q(f)$.
	Assume to the contrary that $L(f_0)$ has an infinite periodic gap $\Omega$ of period $p$ such that $K_{f,\Omega}$ is not a maximal Fatou chain of $f$. By Proposition \ref{pro:ren}, $K_{f,\Omega}$ is a full stable continuum not lying in any maximal Fatou chain. Then Proposition \ref{pro:cluster} yields a regulated arc $\gamma \subseteq K_{f,\Omega}$ that is not contained in any maximal Fatou chain of $f$.

Let $\ell$ and $\ell'$ be two leaves of $L(f)$ with $\pi_f(\ell)=\g(0)$ and $\pi_f(\ell')=\g(1)$. 	
	Applying Proposition \ref{lem:lamination-version} to $\ell,\ell'$, we obtain two leaves $\overline{\theta^-_x\theta_x^+}$ and $\overline{\theta_y^-\theta_y^+}$ of $L(f)$ within $\Omega$, such that  $\pi_f(\theta^\pm_x)=x$ and $\pi_f(\theta^\pm_y)=y$ are distinct points in $K_{f,\Omega}$, and the following properties hold:
	\begin{itemize}
		\item [(a)] the points $x$ and $y$ are pre-repelling and their orbits avoid the critical points of $f$;
		\item [(b)] the component $W$ of $\C\setminus\big(R_{f}(\theta_x^\pm)\cup\{x\}\cup\{y\}\cup R_{f}(\theta_y^\pm)\big)$ whose boundary contains $x$ and $y$
		is disjoint from $P_f$;
		\item [(c)] for $i=1,2$, there exist $m_i>0$ and a component $W^{(i)}$ of $f^{-m_i}(W)$ fullfilling that
		\begin{itemize}
			\item[-]  $W^{(1)}$ and $W^{(2)}$ have disjoint closures;
			\item[-] $W^{(i)}\Subset W$ and $W^{(i)}$ separates the two components of $\partial W$;
			\item[-] $f^{m_i}:W^{(i)}\to W$ is a conformal homeomorphism.
		\end{itemize}
	\end{itemize}
	
	Let $\{f_n\}\subseteq\MMMM_+$ be a sequence of subhyperbolic maps converging to $f$ as $n\to\infty$, such that each $f_n$ is connected to $f_0$ by a finite chain of relative hyperbolic components.
	
	By property (a) and Lemma \ref{lem:stable}, for all sufficiently large $n$, the external rays pairs $R_{f_n}(\theta_x^{\pm})$ and $R_{f_n}(\theta_y^{\pm})$ land at the  points $x_n$ and $y_n$, respectively, with $x_n\to x$ and $y_n\to y$ as $n\to\infty$. There also exist corresponding domains $W_{n}^{(i)}$ and $W_n$ for $f_n$,  perturbed from $W^{(i)}$ and $W$ respectively, such that for $i=1,2$ and all large $n$,  the map  $f^{m_i}_n: W_{n}^{(i)}\to W_n$ is conformal
	and  $W_n^{(i)}\Subset W_n$.
	
	Fix a sufficiently large $n$, and define the map $F_n:W_{n}^{(1)}\cup W_{n}^{(2)}\to W_n$ by $$F_n(z)=f^{m_i}_n(z)\textup{ for }z\in W_{n}^{(i)}, i=1,2.$$ The non-escaping set of $F_n$ contains uncountably many ray pairs landing at $K_{f_n,\Omega}$, while by Proposition \ref{pro:reduction}, $K_{f_n,\Omega}$ is a maximal Fatou chain and hence contains countably many cut points by Proposition \ref{pro:cluster}. This is a contradiction.
\end{proof}

By Theorem \ref{thm:gf} and Claim 1 in its proof, we immediately obtain the following.

\begin{corollary}\label{coro:finite-step}
	Let $\MMMM_+$ be an extended molecule containing a hyperbolic component $\HHHH$. Then every geometrically finite map  $f\in\MMMM_+$ is connected to $\HHHH$ by a finite chain of relative hyperbolic components. If $f$ is hyperbolic, every relative hyperbolic component in this chain is hyperbolic.
\end{corollary}

\section{Extended molecules are molecules}\label{sec:A}


\begin{theorem}\label{thm:same}
	For any bounded hyperbolic component $\HHHH$, we have $\MMMM(\HHHH)=\MMMM_+(\HHHH)$.
\end{theorem}

\begin{proof}
	For simplicity, we write $\MMMM=\MMMM(\HHHH)$ and $\MMMM_+=\MMMM_+(\HHHH)$.
	It is clear that $\MMMM\subseteq\MMMM_+$. Corollary \ref{coro:finite-step} implies that all hyperbolic maps in $\MMMM_+$ belong to $\MMMM$.
	Thus, it suffices to show that any subhyperbolic polynomial in $\MMMM_+$ can be approximated by hyperbolic ones in $\MMMM_+$
	
	Let $g\in\MMMM_+$ be a subhyperbolic map with last critical points $c_1,\ldots,c_m$  in its Julia set. Then $c_1,\ldots,c_m$ are contained in maximal Fatou chains of $g$. Denote by $M$ the number of maximal Fatou chains of $g$ that contain postcritical points.
	Choose a sequence $\{\epsilon_n\}$ of positive numbers such that $\epsilon_n\to 0$ as $n\to\infty$.
	
	Fix a sufficiently large $n$. Applying Theorem \ref{thm:perturbation} to $(g,c_1)$, we obtain a subhyperbolic map $f_{1,n}\in\CCCC_d$ in the $\epsilon_n/m$-neighborhood of $g$, whose Julia set contains $m-1$ last critical points  $c_2(f_{1,n}),\ldots,c_m(f_{1,n})$, which lie in maximal Fatou chains of $f_{1,n}$.
	Inductively, suppose for some $1\leq i\leq m-1$ that $f_{i,n}$ has been constructed. The application of Theorem \ref{thm:perturbation} to $(f_{i,n},c_{i+1}(f_{i,n}))$ yields a subhyperbolic map $f_{i+1,n}$ in the $\epsilon_n/m$-neighborhood of $f_{i,n}$,  whose Julia set contains $m-i-1$ last critical points  $c_{i+2}(f_{i+1,n}),\ldots,c_m(f_{i+1,n})$, which lie in maximal Fatou chains of $f_{i,n}$.
	The resulting map $g_n:=f_{m,n}$ is  hyperbolic, lying in the $\epsilon_n$-neighborhood of $g$. By Theorem \ref{thm:perturbation}\,(4), the number of maximal Fatou chains of $g_n$  containing postcritical points remains $M$.
	
	Applying Proposition \ref{pro:reduction} to $g_n$, we obtain a hyperbolic primitive postcritically finite polynomial $g_n^*$ such that $\# P_{g_{n}^*}=M$. Note that the polynomials $\{g_n^*\}_n$ have only finitely many possibilities.  By passing to a subsequence, we may assume that all $g_n^*$ coincide. It follows that all hyperbolic polynomials $g_n$ lie in a common extended molecule, denoted by $\MMMM_+'$, and therefore $g\in\MMMM_+'$.
	
	Since $g\in\MMMM_+\cap\MMMM_+'$,   Corollary  \ref{coro:finite-step} implies that  $g$ can be connected to the primitive hyperbolic components in both $\MMMM_+$ and $\MMMM_+'$ by finite chains of relative hyperbolic components. By definition, we have $\MMMM_+=\MMMM_+'$. Thus the sequence of hyperbolic maps $\{g_n\}\subseteq \MMMM_+$ converges to $g$.
\end{proof}

\begin{proof}[Proof of Theorem A]
	It follows directly from Theorems \ref{thm:gf} and  \ref{thm:same}.
\end{proof}

\begin{proof}[{Proof of Theorem B}]
	Let $\MMMM$ and $\MMMM'$ be two molecules, with the unique primitive hyperbolic postcritically finite polynomials  $f_0\in\MMMM$ and $g_0\in\MMMM'$, respectively. Suppose that $g\in\MMMM\cap\MMMM'$. Then Theorem A gives $\lambda_\Q(g_0)\subseteq \lambda_\Q(g)$ and $\lambda_\Q(f_0)\subseteq \lambda_\Q(g)$.\vspace{2pt}
	
	We will show that $\lambda_\Q(g_0)\subseteq \lambda_\Q(f_0)$. Due to Lemma \ref{lem:include}, it suffices to prove that  any leaf   of $L_\Q(g_0)$ avoiding gap boundaries is also a leaf of $L_\Q(f_0)$.
Let $\ell=\overline{\alpha\beta}$ be such a leaf of $L_\Q(g_0)$. Suppose on the contrary that $\ell$ is not a leaf of $L_\Q(f_0)$.\vspace{7pt}

{\bf Claim 1}.	\emph{The geodesic  $\ell$ is disjoint from   any leaf in $L(f_0)$}.  
\begin{proof}[Proof of Claim 1]
Assume not, there exists a leaf $\overline{ab}$ of $L(f_0)$ such that  $\overline{ab}$ and $\overline{\alpha\beta}$ either intersects at one point or have exactly one common endpoint. 

If $a$, and hence $b$, is irrational, then by Corollary \ref{coro:same}, $\overline{ab}$ avoids gap boundaries of $L(f_0)$ and  can be approximated by leaves of $L_\Q(f_0)$ in both components of $\D\setminus \overline{ab}$. So we may assume that $\overline{ab}$ is a leaf in $L_\Q(f_0)$. 
 Applying the latter part of Lemma \ref{lem:lamination-version} to $\overline{\alpha\beta}$, there exists  a sequence $\{\overline{\alpha_n\beta_n}\}_{n\geq1}$ of leaves in $L_\mathbb{Q}(g_0)$ with $\overline{\alpha_n\beta_n} \cap \overline{ab} \neq \emptyset$ for all $n$. Since $(\alpha_n,\beta_n), (a,b) \in \lambda_\mathbb{Q}(g)$, all $R_g(\alpha_n),n\geq1$ have to land at a common point, a contradiction. Hence the claim holds.
 \end{proof}
		
By this claim, $\ell$ is contained in an infinite gap $\Omega(f_0)$ of $L(f_0)$. By the latter part of Lemma \ref{lem:lamination-version}, we can choose another leaf $\ell'$ of $L(g_0)$ that intersects $\Omega(f_0)$.  Let $S_0$ be the strip bounded by $\ell$ and $\ell'$. Applying Lemma \ref{lem:lamination-version} to $g_0$ and $S_0$, there exist strips $S'\subseteq S_0,S,S_1$ and $S_2$ in $\D$, satisfying
		\begin{enumerate}
			\item the edges $\overline{\theta\eta}$ and $\overline{\theta'\eta'}$ of $S'$ are leaves of $L_\Q(g_0)$,  both avoiding gap boundaries in $L(g_0)$;\vspace{1pt}
			\item there exists $l>0$ such that $\sigma_d^l:\partial S'\to\partial S$ is a homeomorphism;\vspace{1pt}
			\item $\overline{S_1}\cap \overline{S_2}=\emptyset$, and for each $i\in\{1,2\}$,
			\begin{itemize}
				\item $S_i$ is compactly contained in $S$ and separates its two edges;
				\item  there exists $m_i>0$ such that $\sigma_d^{m_i}:\partial S_i\to\partial S$ is a homeomorphism.
			\end{itemize}
		\end{enumerate}
		
As $g\in\MMMM$, let $\{f_n\}\subseteq\MMMM$ be a sequence of subhyperbolic maps converging to $g$, each connected to $f_0$ by a finite chain of relative hyperbolic components.
Remember that $(\theta,\eta),(\theta',\eta')\in\lambda_\Q(g)$. Since every critical point and parabolic point of $g$ lies in $K_{g,\Omega(g_0)}$ for some infinite gap $\Omega(g_0)$ of $L(g_0)$,
	the orbits of $\pi_g(\theta)$ and $\pi_g(\theta')$ both avoid the critical points and parabolic cycles of $g$. Then by Lemma \ref{lem:stable}\,(1), we have a large $n$ such that 
			\[\text{$\pi_{f_{n}}(\theta)=\pi_{f_{n}}(\eta)=:z$ and $\pi_{f_{n}}(\theta')=\pi_{f_{n}}(\eta')=:z'$}.\]
		Consequently, every edge of $S$, $S_1$ or $S_2$ corresponds to a point in $\lambda_\mathbb{Q}(f_{n})$. Since $\overline{\theta\eta}$ and $\overline{\theta'\eta'}$ separates the two edges of $S_0$, they both intersect $\Omega(f_0)$. Hence $z,z' \in K_{f_{n},\Omega(f_0)}$.  Moreover, we conclude from Claim 1 that  $\overline{\theta\eta},\overline{\theta'\eta'} \subseteq \Omega(f_0)$.
		
		Define the strip $W'$ in the phase space of $f_{n}$ corresponding to $S'$ such that its boundary consists of the arcs $\overline{R_{f_{n}}(\theta)} \cup \overline{R_{f_{n}}(\eta)}$ and $\overline{R_{f_{n}}(\theta')} \cup \overline{R_{f_{n}}(\theta')}$. Similarly define the strips $W$, $W_1$, and $W_2$ corresponding to $S$, $S_1$, and $S_2$ respectively.
		
		By Statement (3) above, we have $\overline{W_1} \cap \overline{W_2} = \emptyset$, and for each $i \in \{1,2\}$:
		\begin{itemize}
			\item $\overline{W_i} \subseteq W$ separates the two boundary components of $W$;
			\item $f_{n}^{m_i}: W_i \to W$ is a conformal homeomorphism.
		\end{itemize}
		Therefore, as shown in the proof of Theorem~\ref{thm:gf}, $f_{n}$ admits uncountably many external ray pairs that land at pairwise distinct common points and separate the boundary components of $W$.
		
		By Statement (2), $f_{n}^l: W' \to W$ is a homeomorphism. Combined with the preceding discussion, it follows that $K_{f_{n},\Omega(f_0)}$ contains uncountably many cut points. Since $K_{f_{n},\Omega(f_0)}$ is a maximal Fatou chain of $f_{n}$ by Theorem A, this contradicts Proposition~\ref{pro:cluster}.	
	So we conclude that $\ell=\overline{\alpha\beta}$ is a leaf of $L(f_0)$, which completes the proof of $\lambda_\Q(g_0)\subseteq \lambda_\Q(f_0)$. 
	
	By symmetry,  $\lambda_\Q(f_0)\subseteq \lambda_\Q(g_0)$ also holds. Therefore, it follows from Proposition \ref{pro:same1} that $f_0=g_0$, proving Theorem B.
\end{proof}

\section{When is a subhyperbolic map on the boundary?}\label{sec:C}

This section presents the proof of Theorem~C. We begin by introducing some necessary notations and preliminary results.

Let $f\in\CCCC_d$ be a subhyperbolic polynomial. An open arc $\g$ is called a \emph{ray pair} of $f$ if
\[\g=R_f(\theta_+)\cup\{z\}\cup R_f(\theta_-),\]
where $R_f(\theta_\pm)$ are distinct external rays landing at the same point $z$. This point $z$ is called the \emph{landing point of $\g$}.

Let $W$ be a component of $\C\setminus \g$. A disk $\De\subseteq W$ is called   a \emph{local sector} of  $W$ if $\De=D\cap W$ for some  disk $D$ that contains the landing point of $\g$. Furthermore, if
$E\subseteq K_f$ is a full continuum satisfying
\begin{itemize}
	\item the landing point of $\g$ belongs to $E$;
	\item $W$ is disjoint from $E$,
\end{itemize}
we say that $W$ is the \emph{wake of $E$ bounded by $\g$}.

\begin{lemma}\label{pro:wake}
	Let $f\in\CCCC_d$ be a subhyperbolic map. Let $K$ be the maximal Fatou chain of $f$ generated by a fixed Fatou domain of $f$. Suppose that $W$ is the wake of $K$ bounded by a ray pair $\g$, then for every integer $n\geq 1$,
	\begin{enumerate}
		\item $f^n(\g)$ is  a ray pair that bounds a wake of $K$, denoted by $W_n$;
		\item $f^n$ sends a local sector of $W$ homeomorphically onto a local sector of $W_n$;
		\item the ray pair $\g$ is eventually periodic under iteration by $f$.
	\end{enumerate}
\end{lemma}
\begin{proof}
	We proceed by induction on $n$; the base case $n=1$ suffices.
	
	Let $\g=R_f(\theta)\cup\{z\}\cup R_f(\theta')$ and let $m={\rm deg}(f, z)\geq 1$. Choose $m$ external rays $R_f(\theta)=R_f(\theta_1),\ldots,R_f(\theta_m)$ landing at $z$, such that $f(R_f(\theta_i))=R_f(\sigma_d(\theta))$ for all $i$. Since $K$ is stable, every component of $$\C\setminus\bigcup_{i=1}^m \ov{R_f(\theta_i)}$$ intersects $K$.  As $K\cap W=\emptyset$, the wake $W$ is properly contained in one such component. It follows that $\theta'\notin\{\theta_1,\ldots, \theta_m\}$. So $\g_1:=f(\g)$ is again a ray pair, and $f$ sends a local sector of $W$ homeomorphically onto a local sector of $W_1$,  a component of $\C\setminus\g_1$.
	
	To show that $W_1$ is a wake of $K$, suppose otherwise. Then there exists a small open arc $\beta\subseteq K\cap W_1$ ending at $f(z)$. Lifting $\beta$ via $f$ yields an open arc $\tilde{\beta}$ in a local sector of $W$ ending at $z$. Since $K$ is stable, we have $\tilde{\beta}\subseteq K$. This, however, contradicts the assumption that $W\cap K=\emptyset$. This completes the proof of Statements (1) and (2).
	
	To prove Statement (3), we show that $z$ is preperiodic.  For each $n\geq1$, define the \emph{width} of $W_n$ as the length of the interval $$\{\alpha\in\mathbb{R}/\mathbb{Z}: R_f(\alpha)\subseteq W_n\}.$$
	
	If $|W_n|<1/d$, then $f(W_n)=W_{n+1}$ and $|W_{n+1}|=d|W_n|$. Hence, for each $n$ there exists $m_n>n$ such that $|W_{m_n}|\geq 1/d$, implying $\sum\limits_{n}|W_n|=+\infty$.

	If $z$ were not preperiodic, the wakes $W_n$ would be pariwise disjoint, so $\sum_{n}|W_n|\leq 1$, a contradiction. Therefore, $z$ is preperiodic, and thus $\g$ is eventually periodic.
\end{proof}
Next, we show that the property of critical points being iterated into maximal Fatou chains is invariant under dynamical perturbation given by Theorem  \ref{thm:perturbation}.
\begin{proposition}\label{pro:same}
	Let $f\in\CCCC_d$ be a subhyperbolic map, and $c^*$ be a last critical point of $f$ lying on the boundary of a bounded Fatou domain. Let $g$ be a dynamical perturbation of $(f,c^*)$ as given in Theorem \ref{thm:perturbation}. If every critical point of $f$ is eventually mapped into the maximal Fatou chain generated by a periodic bounded Fatou domain, then the same property still holds for $g$.
\end{proposition}

\begin{proof}
	Assume, by contradiction, that there exists a critical point
	$c_g$ of $g$ that is never iterated into any maximal Fatou chain generated by a bounded Fatou domain.
	
	By Theorem \ref{thm:perturbation}\,(5), there exists a sequence $\{g_n\}$ of polynomials, obtained via quasiconformal deformations of $g$ such that $g_n\to f$ as $n\to\infty$.
	For notational simplicity, if $A_g$ denotes an object associated with $g$ (e.g. an external ray, a preperiodic point, etc.), we write $A_n$ for the corresponding object under $g_n$, and $A:=\lim\limits_{n\to\infty} A_n$ for the corresponding object under $f$.
	
	By the hypothesis of the proposition, there exists a periodic bounded Fatou domain $ U$ of $f$ and an integer $l\geq0$ such that $f^l(c)$ belongs to the maximal Fatou chain $E$ generated by $ U$. Without loss of generality, we may assume that $f( U)= U$.
	
	Let $ U_g$ be the attracting Fatou domain of $g$ corresponding to $ U$, and let $E_g$ be the maximal Fatou chain generated by $ U_g$. Fix an external ray $R_g(\alpha)$ lands at $g^l(c_g)$. By assumption, $g^l(c_g)\notin E_g$. So  $\overline{R_g(\alpha)}$ is contained in a wake $W_g$ of $E_g$ bounded by a ray pair $\g_g=R_g(\theta)\cup \{z_g\}\cup R_g(\theta')$. Since $\g_g$ is $g$-preperiodic by Lemma \ref{pro:wake}, we may assume that
	$\g_g':=g(\g_g)$ is fixed by $g$. Then $w_g:=g(z_g)\in E_g$ is a repelling fixed point of $g$. Denote by $W_g'$ the wake of $E_g$ bounded by $\g'_g$. 

	By Lemma \ref{lem:stable}\,(1), $\g_n\to \g=R_f(\theta)\cup\{z\}\cup R_f(\theta')$ as $n\to\infty$. Then $\g$ separates $R_f(\alpha)$ from $ U$. We now verify the following two properties:
	\begin{enumerate}
		\item $f$ sends a local sector of $W$ homeomorphically onto a local sector of $W'$;
		\item $\overline{R_f(\alpha)}$ and $ U$ lie in distinct components of $\mathbb{C}\setminus\g.$
	\end{enumerate}
	
	To show (1), let $d_0:={\rm deg}(g, z_g)\geq 1$. Then there exists $d_0$ external rays $$R_g(\theta)=R_g(\theta_1),\ldots, R_g(\theta_{d_0})$$
	landing at $ z_g$, such that $g$ maps each of them onto $R_g(\sigma_d(\theta))$. Since $g$ maps a local sector of $W_g$ homeomorphically to that of $W_g'$, the wake $W_g$ lies entirely in one component of $$\C\setminus\bigcup_{j=1}^{d_0}\overline{R_g(\theta_j)}.$$
	
	By Theorem \ref{thm:perturbation}, ${\rm deg}(f,  z)=d_0$. Combined with Lemma \ref{lem:stable}, the $d_0$ rays $R_f(\theta_j)$ are all preimages under $f$ of $R_f(\sigma_d(\theta))$ that land at $z$. Since $W$ lies in a component of $\C\setminus\bigcup_{j=1}^{d_0}\overline{R_f(\theta_j)}$, property (1) holds.\vspace{2pt}
	
	To show (2), it suffices to prove that $z$ is not the landing point of $R_f(\alpha)$. Suppose otherwise; then $R_f(\sigma_d(\alpha))$ lands at the repelling fixed point $w$ of $f$. Here $w=\lim_{n\to\infty} w_n$ and $\{w\}=\gamma'\cap J_f$. Since $R_f(\sigma_d(\alpha))\subseteq W'$, it follows from Lemma \ref{lem:stable} that $R_g(\sigma_d(\alpha))\subseteq W_g'$ and  it lands at $w_g$. Hence, there exists $\beta\neq \alpha$ such that $\sigma_d(\beta)=\sigma_d(\alpha)$ and $R_g(\beta)\subseteq W_g$ lands at $z_g$.  Consequently, $R_f(\beta)\subseteq W$ lands at $z$.
	Then the two distinct rays $R_f(\alpha)$ and $R_f(\beta)$ in $W$  are both mapped onto $R_f(\sigma_d(\alpha))$ by $f$, contradicting property (1).\vspace{2pt}

	We now claim that the fixed point $w$ is a cut point of the maximal Fatou chain $E$ generated by $U$. Indeed,  note first that $R_f(\alpha)$ lands at $f^l(c)$. Since both $f^l(c)$ and $ U$ are contained in $E$, by properties (1) and (2), we may choose a point $z_*\in E\cap W$ in a local sector of $W$ such that $f(z_*)\in E\cap W'$. Then $\g'$ separates $f(z_*)$ from $ U$. So $w$ disconnects $E$.
	
	This claim contradicts Proposition \ref{pro:chain1}. So the  proposition is proved.
\end{proof}

\begin{proof}[Proof of Theorem C]
	
	Suppose $J_f\cap C_f\neq \emptyset$ and each critical point of $f$ is iterated to the maximal Fatou chain generated by a periodic bounded Fatou domain. By Proposition \ref{pro:chain1}, at least one critical point lies on the boundary of a bounded Fatou domain.
	
	Applying Theorem \ref{thm:perturbation}  and Proposition \ref{pro:same} successively, we obtain a sequence of subhyperbolic polynomials $f=f_0,\ldots,f_m\in \CCCC_d$
	such that, for each $i=1,\ldots,m$,
	\begin{enumerate}
		\item the map $f_{i}$ is a dynamical perturbation of $(f_{i-1},c_{i-1})$, where $c_{i-1}$ is a last critical point of $f_i$ lying on the boundary of a bounded Fatou domain;
		\item each critical point of $f_i$ is iterated to the maixmal Fatou chain generated by a periodic bounded Fatou domain;
		\item the map $f_m$ has no critical points on boundaries of its bounded Fatou domains.
	\end{enumerate}
	It follows that $f_m$ is hyperbolic, and belongs to a bounded hyperbolic component $\HHHH$.
	
	By Theorem \ref{thm:perturbation}\,(5), for each $i=1,\ldots,m$, there exists a sequence $\{f_{i,n}\}_{n}$ of subhyperbolic maps such that each $f_{i,n}$ (for $n\geq1$) is quasiconformal conjugate to $f_i$, and $f_{i,n}\to f_{i-1}$ as $n\to\infty$. In particular, $\{f_{m,n}\}_{n}\subseteq\HHHH$ and hence $f_{m-1}\in\partial \HHHH$.
	
	According to \cite[Theorem 1.1]{GYZ}, if a subhyperbolic rational map $g_*$ lies in the boundary of a hyperbolic component $\HHHH_*$, then any quasiconformal deformation of $g_*$ remains in $\partial \HHHH_*$. Hence $\{f_{m-1,n}\}_{n}\subseteq\partial\HHHH$, and thus $f_{m-2}\in\partial\HHHH$. By induction, it follows that $f\in\partial\HHHH$.
	\vspace{3pt}
	
	For the necessity, let $\HHHH$ be a bounded hyperbolic component with a sequence $\{f_n\}$ converging to $f\in\partial\HHHH$. Fix a critical point $c$ of $f$. It then suffices to show that $c$ is eventually mapped into a maximal Fatou chain generated by a periodic Fatou domain.

	Let $c_n$ be a critical point of $f_n$ with $c_n\to c$ as $n\to \infty$.
	Since $f_n\in\HHHH$, there exists a periodic attracting Fatou domain $ U_n$ and a minimal integer $l\geq0$ such that $f_n^l(c_n)\subseteq  U_n$. Without loss of generality, we may assume $f_n( U_n)= U_n$ for all $n\geq1$. Let $z_n$ be the attracting fixed point in $ U_n$. Then $z:=\lim_{n\to\infty} z_n$ is an attracting fixed point of $f$.
	Let $ U$ be the Fatou domain of $f$ containing $z$, and $K$ be the maximal Fatou chain of $f$ generated by $ U$.
	By Proposition \ref{pro:renormalization}, given any $\epsilon>0$, there exist nested disks $W\Subset V$ containing $K$ and lying in an $\epsilon$-neighborhood of $K$, such that $f:W\to V$ is a branched covering of degree ${\rm deg}(f|_K)$.
	
	We  claim that for every sufficiently large $n$, the entire Fatou domain $ U_n$ is contained in $V$. To prove this claim,  choose a small neighborhood $D\subseteq U$ of $z$ such that $f(D)\Subset D$. Then there exists $N>0$ such that $z_n\in D\subseteq  U_n$ and $f_n(D)\Subset D$ for all $n>N$. Moreover, by increasing $N$ if necessary, we may assume that for each $n>N$, there exists a disk $W_n\Subset V$ such that $f_n:W_n\to V$ is a branched covering of  degree ${\rm deg}(f|_K)$.
	
	Note that $ U_n=\bigcup_{k\geq1} D_n^k$, where $D_n^k$ is the component of $f^{-k}_n(D)$ containing $z_n$. Since $D\subseteq V$ and $f_n:W_n\to V$ is a branched covering, it follows that $D_n^1\subseteq W_n\Subset V$. Inductively, each $D_n^k$ is contained in $V$. Thus, $ U_n\subseteq V$. This proves the claim.
	
	From the claim, we have $f_n^l(c_n)\in V$ for all $n>N$. This implies $f^l(c)\in\ov{V}$. Since $V$ is contained in an $\epsilon$-neighborhood of $K$ and $\epsilon>0$ is arbitrary, it follows that $f^l(c)\in K$. This completes the proof of the necessity.
\end{proof}



\end{document}